\documentclass[a4paper,11pt, twoside]{article}
\usepackage{amsmath,amsthm}
\usepackage{amssymb,latexsym}
\usepackage{mathrsfs}
\usepackage{enumerate}
\usepackage[colorlinks,citecolor=blue]{hyperref}
\usepackage[numbers,sort&compress]{natbib}
\usepackage{indentfirst}
\usepackage{verbatim}
\usepackage{color}
\usepackage{enumitem}
\DeclareMathAlphabet{\mathpzc}{OT1}{pzc}{m}{it}

\newtheorem{theorem}{Theorem}[section]

\newtheorem{Lemma}{Lemma}[section]
\newtheorem{proposition}{Proposition}[section]

\newtheorem{remark}{Remark}[section]

\theoremstyle{definition} \theoremstyle{remark}
\numberwithin{equation}{section}
\allowdisplaybreaks

\date{}

\begin{document}
	
	\markboth{J.W. He, S.L. Li, Y. Zhou}{ Maximal $L_p$-regularity for fractional   problem...}
	
	\date{}
	\baselineskip 0.22in
	\title{{\bf Maximal $L_p$-regularity for fractional problem driven by non-autonomous forms}}
	
	\author{Jia Wei He$^{1, }$\footnote{Corresponding author,  jwhe@gxu.edu.cn}, Shi Long Li$^1$, Yong Zhou$^{2,3}$ \\[1.7mm]
		\footnotesize  {1. School of Mathematics and Information Science, Guangxi University, Nanning 530004, China}\\ 
		\footnotesize  {2. Faculty of Mathematics and Computational Science, Xiangtan University, Hunan 411105,  China}\\
		\footnotesize  {$^3$Faculty of Information Technology, Macau University of Science and Technology, Macau 999078, China} 
	}

	\maketitle
	
	\begin{abstract}
	We investigate the maximal $L_p$-regularity in J.L. Lions' problem  involving a time-fractional derivative and a non-autonomous form $a(t;\cdot,\cdot)$ on a Hilbert space $H$. This problem says whether the maximal $L_p$-regularity in $H$ hold when $t \mapsto a(t ; u, v)$ is merely continuous or even merely measurable. We prove the maximal $L_p$-regularity results when the coefficients satisfy general Dini-type continuity conditions. In particular, we construct a counterexample to negatively answer this problem, indicating 
  the minimal  H\"{o}lder-scale regularity required for positive results.  \\[2mm]
	\end{abstract}

	\baselineskip 0.25in
	\section{Introduction}
	
	Let $\{-A(t)\}_{t\in[0,\tau]}$ be a family of  generators of analytic semigroups on some Banach space $X$ for some $\tau\in(0,\infty)$. It is well known that,  for the non-autonomous Cauchy problem
	\begin{equation}
		\label{P-1}\tag{P}
		\left\{
		\begin{aligned}
			u'(t)+A(t)u(t)&=f(t),~~ \text{ for~a.e. }~~  t\in (0,\tau),
			\\ u(0)&=u_0,
		\end{aligned}\right.
	\end{equation}
a classical J.L. Lions' problem asks whether the maximal regularity holds involving non-autonomous symmetric forms when the time dependence is merely continuous or even merely measurable.
	More generally, the maximal $L_p$-regularity problem with $p\in (1,\infty)$ concerns whether a unique solution exists for any given $f\in L_p(0, T; X)$ within the maximal regularity space:
	\begin{align*}
		MR_p(0,T;X):=\{u\in W^{1,p}(0,T;X):\ \ & u(t)\in D(A(t)),\ \ {\rm a.e.}\ t\in[0,T],\\ 
		&{\rm and}\ A(\cdot)u(\cdot)\in L_p(0,T;X)\}.
	\end{align*} 
	
	This problem has attracted increasing attention in recent years, with ongoing and notable developments shaping the field. In the autonomous setting, i.e., when $A(t) = A$ independent of time $t$, a well-known result by Simon \cite{Simon} stated that the maximal $L_p$-regularity holds in  Hilbert space if $-A$ generates a bounded analytic semigroup. 
	The situation has significant differences in Banach spaces. Kalton and Lancien \cite{Kalton} pointed out there exist generators of analytic semigroups that fail to possess the maximal $L_p$-regularity in Banach space. Considering $-A$ as a generator of bounded analytic semigroup on a UMD space, Weis \cite{Weis} established the maximal $L_p$-regularity iff the set $\{t (i t + A)^{-1} : t \in \mathbb{R} \backslash \{0\}\}$ is $R$-bounded. We refer \cite{Dore, Kunstmann, Hytonen} for more related works.
	In order to study the initial-boundary value problems of parabolic equations with time dependent boundary conditions on bounded domains, it is natural to derive non-autonomous equations whose operator domains vary with time.  Consider a non-autonomous sesquilinear form $a:[0, \tau] \times V \times V \rightarrow \mathbb{C}$ with uniform boundedness and coercivity, namely for some constants $M>0$ and $\gamma>0$, and for all $t \in[0, \tau]$, $u, v\in V$, the following conditions hold:
	\begin{equation}\label{A0}\tag{A}
		\begin{aligned}
			|a(t, u, v)| \leq M\|u\|_ V\|v\|_{V},\ \ \
			\operatorname{Re} a(t, u, u)\geq \gamma\|u\|_{V}^2,
		\end{aligned}
	\end{equation}
where $V$ is a Hilbert space over $\mathbb{R}$ or $\mathbb{C}$.  Let $H$ be another Hilbert space such that
the Gelfand triple $V \hookrightarrow H \hookrightarrow V^{\prime}$ holds with $V^{\prime}$ being the dual space of $V$. Consider operator $\mathcal{A}(t)\in L(V,V^{\prime})$ given by $\langle\mathcal{A}(t)u,v\rangle=a(t;u,v)$, and $A(t)$ is the part of $\mathcal{A}(t)$ on $H$, a foundational result by Lions \cite{Lions} states that if the form $a(t;u,v)$
 is strongly measurable with $t$, then $\{\mathcal{A}(t)\}_{0\leq t\leq \tau}$ has maximal $L_2$-regularity in $V^{\prime}$.  To establish maximal \( L_p \)-regularity in the dual space $V'$, a natural idea is to improve the additional regularity constraints of the operator family $\{\mathcal{A}(t)\}_{0\leq t\leq \tau}$. Arendt et al. \cite{Arendt} showed that a sufficient condition for this regularity is the (piecewise relative) continuity of the mapping \( t \mapsto \mathcal{A}(t)\) in $ L(V, V')$. However, Bechtel et al. \cite{Bechtel} later provided a critical counterexample demonstrating that maximal \( L_p \)-regularity in \( V' \) fails for \( p \neq 2 \) when \( t \mapsto \mathcal{A}(t)\) is merely strongly measurable. Their analysis underscores the sharpness of the continuity requirement for non-Hilbertian \( L_p \)-spaces. 
  
 	Hieber and Monniaux \cite{Hieber} proposed an effective approach to establish maximal $L_p$-regularity by utilizing a Volterra equation introduced by Acquistapace and Terreni \cite{Acquistapace}, i.e.,
	\begin{equation}\label{Eq-1}
		\begin{aligned}
			u(t)= & \int_0^t e^{-(t-s) \mathcal{A}(t)} f(s) \mathrm{d} s+e^{-t \mathcal{A}(t)} u_0  \\
			& +\int_0^t e^{-(t-s) \mathcal{A}(t)}(\mathcal{A}(t)-\mathcal{A}(s)) u(s) \mathrm{d} s.
		\end{aligned}
	\end{equation}
However, in some elliptic boundary value problems, maximal regularity in $H$ is essential for understanding the behavior of solutions because it enables the effective handling of boundary condition. Based on the framework of (\ref{Eq-1}), Ouhabaz and Spina \cite{Spina} proved maximal  $L_p$-regularity in $H$ when $u(0)=0$ assuming the $1/2+\epsilon$-H\"{o}lder continuity of non-autonomous forms in time, i.e. for $\epsilon>0$, 
	\begin{equation}
		|a(t, u, v)-a(s, u, v)| \leq c(t-s)^{\frac{1}{2}+\epsilon}\|u\|_V\|v\|_V.
	\end{equation}	
	An observation is that every H\"{o}lder function is the Dini type, Haak and Ouhabaz \cite{Haak} introduced a more general Dini type continuity condition
	$$
	|a(t, u, v)-a(s, u, v)| \leq  \omega(t-s)\|u\|_V\|v\|_V,
	$$
	 to establish the maximal $L_p$-regularity for $u_0=0$
     and further for $u_0\in {(H, \mathcal{D}(A(0)))_{1-\frac{1}{p}, p}}$ if $\omega$ satisfies a $p$-Dini condition. 
By constructing a counterexample,  Fackler \cite{Fackler2} showed that the exponent $1/2+\epsilon$ is optimal, and there exists a symmetric form with $C^{1/2}$-regularity such that the maximal regularity in $H$ fails. An alternative regularity condition associated with H\"{o}lder scale is the fractional Sobolev regularity, Dier and Zacher \cite{Dier} proved maximal $L_2$-regularity in $H$ if  $\mathcal{A}(t)$ belongs to fractional Sobolev space   $H^{1/2+\epsilon}(0,\tau;L(V,V'))$.
Achieving maximal regularity in the critical space $H^{1/2}$ is relatively difficult, in fact, if  $\mathcal{A}(t)\in W^{1/2,p}(0,\tau;L(V, V'))$ for $p<2$, it fails to imply maximal $L_2$-regularity, see e.g. \cite{Dier2014},  since the form does not satisfy Kato's square root condition. The generalization is also discussed in \cite{Arendt5},  and whether maximal regularity holds in the critical case has been raised as an open problem.  
By means of an integrability condition, Achache and Ouhabaz \cite{Achache2} solved this open problem  involving maximal $L_2$-regularity in $H$ for $\mathcal{A}(t)\in H^{1/2}(0,\tau;L(V,V^{'}))$, provided that $\mathcal{A}(t)$ satisfies the uniform Kato square root condition. This result is optimal since $C^{1/2}\subset H^\alpha$ for $\alpha<1/2$. Note that in UMD spaces, if $A(t)$ satisfies the Acquistapace-Terreni condition, Portal and \v{S}trkalj \cite{Portal} established maximal $L_p$-regularity  based on  the $R$-boundedness of set $\{t (i t+ A(t))^{-1}: t \in \mathbb{R} \backslash\{0\}\}$. Besides, Fackler \cite{Fackler} obtained maximal $L_p$-regularity using fractional Sobolev regularity. For more related works, see \cite{ Bardos, Arendt5} and the references therein.

	Consider the following fractional Cauchy problem driven by non-autonomous forms
	\begin{equation}
		\label{P-2}\tag{FP}
		\left\{
		\begin{aligned}
			\partial_t^{\alpha} u(t)+A(t)u(t)&=f(t), \text{ for~a.e. }~~  t\in (0,\tau),
			\\ u(0)&=u_0,
		\end{aligned}\right.
	\end{equation}
	where $\partial_t^\alpha$ is the standard Caputo fractional derivative of order $\alpha\in(0,1)$. It is natural to ask whether this fractional problem exists a unique solution for given $f\in L_p(0, T; X)$ within the maximal regularity space:
	\begin{align*}
		MR_p^{\alpha}(0,\tau;X):=\{\partial_t^{\alpha}u\in L_p(0,\tau;X):\ \ & u(t)\in D(A(t)),\ \ {\rm a.e.}\ t\in[0,\tau],\\ 
		&{\rm and}\ A(\cdot)u(\cdot)\in L_p(0,\tau;X)\}.
	\end{align*} 
 Fractional Cauchy problems have garnered increasing attention in recent years due to their remarkable ability to describe the initial state for various anomalous diffusion processes and memory effects in complex systems, which also serve as effective models established in heterogeneous media, viscoelastic materials, and turbulent flows, reflecting the non-local and history-dependent nature of the systems, see \cite{APP1,APP2,APP3} and the references therein. 
 
Maximal $L_p$-regularity  in UMD spaces for (FP) was established  by Pr\"{u}ss \cite{Jan} in one of the earlier studies, where autonomous operator $-A$ possesses a bounded imaginary power. 
 Using the $R$-boundedness, 
Bazhlekova \cite{Bazhlekova} also given a maximal $L_p$-regularity result in UMD spaces. Zacher \cite{Zacher1,Zacher2} generalized these results and characterized maximal $L_p$-regularity allowing nonzero initial data. In the non-autonomous framework, Zacher \cite{Zacher} proved maximal $L_2$-regularity in $V^{\prime}$ assuming the form is strongly measurable with respect to time $t$. By enhancing the assumption of time regularity for the non-autonomous form in  \cite{Zacher}, Achache \cite{Achache3} extended the result to maximal $L_p$-regularity in $V'$ under the assumption that the form is continuous with time $t$. Recently, Achache \cite{Achache} established maximal $L_p$-regularity in $H$ for $u_0\in Tr^{p}_{\alpha}$ under the condition (A) and the $\alpha/2+\epsilon$-H\"{o}lder regularity condition
	\begin{equation}\label{EQ-2}
		|a(t, u, v)-a(s, u, v)| \leq c(t-s)^{\frac{\alpha}{2}+\epsilon}\|u\|_V\|v\|_V,
	\end{equation}
	where $Tr^{p}_{\alpha}=H$ for $\alpha<\frac{1}{p}$ and $Tr^{p}_{\alpha}=\{0\}$ for $\alpha \geq \frac{1}{p}$.
		
		Inspired by the aforementioned works, two interesting problems naturally arises:
		
{\em  Does maximal $L_p$-regularity in $H$ hold for $({FP})$ when $t \rightarrow a(t ; u, v)$ is merely continuous or even merely measurable? If not, what is the minimal regularity condition required?
}

		One of the aims of the present work is to improve upon the results of Achache \cite{Achache} by establishing maximal $L_p$-regularity in $H$ under a more general Dini scale condition, allowing $u_0 \neq 0$ for $\alpha \geq \frac{1}{p}$. Another aim is to provide a counterexample demonstrating that the order of the Dini scale is optimal. We note that our approach to establishing maximal regularity is quite different from \cite{Achache}, where the author uses the Dore-Venni type theorem for non-commuting operators studied in \cite{Monniaux} to obtain the maximal regularity. Specifically, we present an interesting abstract Volterra equation of solution to the non-autonomous evolution problem (FP), namely
		$$
		\begin{aligned}
			u(t)=&\int_{0}^{t}\psi_{A(t)}(t-\tau)f(\tau)d\tau+\phi_{A(t)}(t)u_0\\
			&+\int_{0}^{t}\psi_{A(t)}(t-\tau)(A(t)-A(\tau))u(\tau)d\tau,
		\end{aligned}
		$$
		where $\psi_{A }$ and $\phi_{A }$ are operators generated by $A(t)$ using functional calculus, as discussed in Section \ref{sect:2} below.  This abstract Volterra equation enables us to transform the maximal $L_p$-regularity problem into $L_p$-boundedness of the integral operator with an operator-valued kernel.
		Our observation is that due to non-local nature of the fractional derivative in time, this abstract Volterra equation naturally exhibits many properties that differ from those of previous first-order equations, likely the semigroup property is not satisfied for  $\psi_{A } $ and $\phi_{A} $. This also means that the behaviors of the non-autonomous fractional problem (FP) are markedly different from that of the  problem (P). Therefore, the Dini scale and initial conditions required for maximal regularity are also different.
		
		In the following, we will describe our main conclusions in more detail.
		\begin{theorem}\label{T1}
			Assume that the non-autonomous forms $\{a(t , \cdot, \cdot)\}_{0 \leq t \leq \tau}$ satisfy the hypothesis (\ref{A0}) and the regularity condition
			$$
			|a(t, u, v)-a(s, u, v)| \leq \omega(|t-s|)\|u\|_V\|v\|_V
			$$
			where $\omega:[0, \tau] \rightarrow[0, \infty)$ is a non-decreasing continuous function such that
			\begin{equation}\label{A1}
				\int_0^\tau \frac{\omega(t)}{ t^{1+\frac{\alpha}{2}}  } d  t<\infty.
			\end{equation}
			Then the problem (FP) with $u_0\in Tr^{p}_{\alpha}$ has maximal $L_p$-regularity in $H$ for all $p \in(1, \infty)$. When $\alpha>\frac{1}{p}$, If $\omega$ additionally satisfies the $(\alpha, p)$-Dini condition, i.e.,
			\begin{equation}\label{A2}
				\int_0^\tau\left(\frac{\omega(t)}{t^{\alpha}}\right)^p \mathrm{~d} t<\infty,
			\end{equation}
			then (FP) has maximal $L_p$-regularity for all initial values $u_0 \in(H, \mathcal{D}(A(0)))_{1-\frac{1}{\alpha p}, p}$. Moreover there exists a positive constant $C$ such that
			\begin{equation} \label{Est}
				\begin{aligned}
					\|u\|_{L_p(0, \tau ; H)}&+\|\partial_t^{\alpha}u\|_{L_p(0, \tau ; H)}+\|A(\cdot) u \|_{L_p(0, \tau ; H)}\\  
					\leq &C(\|f\|_{L_p(0, \tau ; H)}+\|u_0\|_{(H, \mathcal{D}(A(0)))_{1-\frac{1}{\alpha p}, p}}).
			\end{aligned}\end{equation}
			In particular, problem (FP) has maximal $L_p$-regularity for all $u_0 \in(H, \mathcal{D}(A(0)))_{\epsilon, 2}$ with any $\epsilon>0$ when $\alpha=\frac{1}{p}$.
		\end{theorem}
		
		\begin{remark}
			The first assertion of this theorem extends the result of Achache \cite[Theorem 3.10]{Achache}, since the condition (\ref{EQ-2}) is a special case of our more general condition (\ref{A1}). The second assertion of this theorem provides an maximal $L_p$-regularity result that allows nonzero initial data $u\neq0$ when $\alpha\geq\frac{1}{p}$.
		\end{remark}
		
		\begin{remark}
			We compare our results for $\alpha$-order fractional Cauchy problem with the results of Haak and Ouhabaz \cite{Haak} for first-order Cauchy problem. Specifically, if
			$$
			\int_0^\tau \frac{\omega(t)}{t^{3 / 2}} \mathrm{~d} t<\infty
			$$
			then the Cauchy problem (P) with $u(0)=0$ has maximal $L_p$-regularity in $H$ for all $p \in$ $(1, \infty)$. If $\omega$ additionally satisfies the condition
			$$
			\int_0^\tau\left(\frac{\omega(t)}{t}\right)^p \mathrm{~d} t<\infty
			$$
			then $(P)$ has maximal $L_p$-regularity for all initial value $u_0 \in(H, \mathcal{D}(A(0)))_{1-1 / p, p}$. We observe that the conditions required for the initial value $u_0$ and the Dini scale  $\omega(t)$ of the form to achieve maximal $L_p$-regularity in the  problem (FP) are weaker than those of the classical one.
		\end{remark}

		Now we provide a sketch of the proof. According to the definition of the maximal regularity space, a solution to the  problem (FP) possesses maximal $L_p$-regularity in $H$ if and only if $A(t) u(t) \in L_p(0, \tau ; H)$. To address this, we introduce transformations for a sufficiently large $\delta>0$: 
		$u^{\delta}(t)=u(t)e^{-\delta t}$, $\phi^{\delta}_{\mathcal{A}(t)}(s)=\phi_{\mathcal{A}(t)}(s)e^{-\delta s}$, $\psi^{\delta}_{\mathcal{A}(t)}(s)=\psi_{\mathcal{A}(t)}(s)e^{-\delta s}$, $f^{\delta}(t)=f(t)e^{-\delta t}$, the abstract Volterra equation becomes
		\begin{equation}\label{EQ4.1}
			\begin{aligned}
				u^{\delta}(t)=&\phi^{\delta}_{\mathcal{A}(t)}(t)u_0+\int_{0}^{t}\psi^{\delta}_{\mathcal{A}(t)}(t-s)f^{\delta}(s)ds
				\\&+\int_{0}^{t}\psi^{\delta}_{\mathcal{A}(t)}(t-s)(\mathcal{A}(t)-\mathcal{A}(s))u^{\delta}(s)ds.
			\end{aligned}
		\end{equation}
		Clearly it suffices to prove $\mathcal{A}(t)u^{\delta}(t)\in L_p(0,\tau; H)$. We can express $\mathcal A(\cdot)u^{\delta}$ as an operator sum $A(\cdot) u^\delta=R u_0+L f+Q\left(A(\cdot) u^\delta\right)$
		where
		$$
		\begin{gathered}
			(Q g)(t):=\int_0^t \mathcal{A}(t) \psi_{\mathcal{A}(t)}^\delta(t-s)(\mathcal{A}(t)-\mathcal{A}(s)) \mathcal{A}(s)^{-1} g(s) d s, \\
			(L f)(t)=\mathcal{A}(t) \int_0^t \psi_{\mathcal{A}(t)}^\delta(t-s) f(s) d s,~  \text{ and }~ \left(R u_0\right)(t)=\mathcal{A}(t) \phi_{\mathcal{A}(t)}^\delta(t) u_0.
		\end{gathered}
		$$
		Next, we establish the $L_p$-boundedness and invertibility of the operator $I-Q$, where the transformation (\ref{EQ4.1}) and condition (\ref{A1}) play crucial roles, which allow us to transform the $A(\cdot) u^\delta$ into
		$$
		A(\cdot) u^\delta=(I-Q)^{-1}\left(R u_0+L f^\delta\right).
		$$
		Consequently, the proof reduces to showing 
		$$
		Ru_0,Lf\in L_p(0,\tau;H).
		$$
		This is achieved by applying the pseudo-differential operator theorem. We transform the operator $L$ into a pseudo-differential operator whose symbol is operator-valued. Then we establish $L_2$-boundedness for this operator, which allows us to extend the result to $L_p$-boundedness by verifying a H\"{o}rmander integral condition for the integral operator $L$. The proof that $Ru_0\in L_p(0,\tau;H)$ follows from a classical operator estimate where condition ($\ref{A2}$) is used.

		The second main result of this paper is as follows. We construct a counterexample with an ${\alpha}/{2}$-H\"{o}lder continuous form,  providing an answer to the problem stated above: maximal $L_p$-regularity in $H$ fails for $(\mathrm{FP})$ when $t \rightarrow a(t ; u, v)$ is merely continuous,  and the order of our regularity condition (\ref{A1}) is optimal.
		\begin{theorem}\label{T2}
			Given any $\tau \in(0, \infty)$, there is a Gelfand triple $V \hookrightarrow H \hookrightarrow V^{\prime}$, if non-autonomous form $a(\cdot;\cdot,\cdot):[0, T] \times V \times V \rightarrow \mathbb{C}$ satisfies (\ref{A0})  and 
			$$
			|a(t, u, v)-a(s, u, v)| \leq K|t-s|^{\frac{\alpha}{2}}\|u\|_V\|v\|_V \quad \text { for all } t, s \in[0, T], u, v \in V,
			$$
		for some $K \geq 0$, then the  problem (FP) fails having maximal $L_p$-regularity in $H$ for the initial value $u_0= 0$ and some $f$ in $L_2(0, T ; H)$.
		\end{theorem}

		\begin{remark}
			Inspired by Fackler's method \cite{Fackler2} of constructing counterexamples for  problem (P), we employ the following strategy to develop our counterexample. We first consider the function $u(t,x)=c(x)E_{\alpha,1}(i^{\alpha}t^{\alpha}\varphi(x))$ such that $u\notin MR^{\alpha}_2 (0,T;H)$, where $E_{\alpha,1}(z)$ is the Mittag-Leffler function. This function allows us to identify a non-autonomous form that satisfies the uniform boundedness and coercivity conditions, and also satisfies equation (FP). Furthermore, it ensures that the form possesses $ {\alpha}/{2}$-H\"older continuity. This leads to our counterexample. 
		\end{remark}

		The paper is organized as follows.  In  Section \ref{sect:2}, we introduce the necessary notations and fundamental properties of the operators $\psi_{A } $ and $\phi_{A} $ generated by $A(\cdot)$ and give the representation formula for the solutions. In Section \ref{sect:4}, we first prove several key lemmas and then establish the maximal $L_p$-regularity of solutions.  In Section \ref{sect:5}, we construct a counterexample to demonstrate the optimality of the H\"older condition.

		\section{Preliminary} \label{sect:2}
		In this section, we first introduce some notations that will be useful throughout this paper. Let $\Gamma(\cdot)$ be the Gamma function and define function $k(t):= t^{-\alpha}/{\Gamma(1-\alpha)}$ for $\alpha\in(0,1)$, $t>0$. The Caputo fractional derivative is defined as 
		$$
		\partial_t^{\alpha}u(t)=\frac{d}{dt}k*(u(\cdot)-u(0))(t)=\frac{d}{dt}\int_0^t k(t-s)(u(s)-u(0))ds,
		$$
		for $u\in L_1(\mathbb{R}^{+})$, where $*$ denotes the convolution. If $u\in H^{1}(\mathbb R^+)$, it can also be expressed as
		$$
		\partial_t^{\alpha}u(t)=k*u'(t).
		$$
		
		Let $\vartheta\in(0,\pi)$. Denote by $\Sigma_\vartheta$ the sectorial domain 
		$$
		\Sigma_\vartheta=\{ z \in \mathbb{C}\setminus\{0\} :~~|\text{arg}(z)| < \vartheta \}.
		$$
		
		Denote  by $(\cdot,\cdot)_Y$ the the scalar product of $Y\in\{V,V^{\prime}, H\}$ and $\langle\cdot,\cdot\rangle_{V,V^{\prime}}$ the dual pairing $V\times V^{\prime}$. In the following, we always assume that $C>0$ is a generic constant.
		
		\begin{proposition}\label{proposition2.1}
		Operators $A(t)$ and $\mathcal{A}(t)$ associated with non-autonomous form satisfying assumption (\ref{A0}) are the sectorial operators, and for some $\theta\in(\frac{\pi}{2},\pi)$, the  properties are true:
			\begin{enumerate}[label=(\roman*)]
				\item [{\rm(i)}]	$
				\rho(-\mathcal{A}(t))\supset \Sigma_{\theta} \cup\{0\}. $
				\item [{\rm(ii)}] $
				\|(\lambda+A(t))^{-1}\|_{L(H)}\leq \frac{C}{|\lambda|}
				$,   for  $\lambda\in\Sigma_{\theta}$.
				\item [{\rm(iii)}]$
				\|(\lambda+\mathcal{A}(t))^{-1}\|_{L(V^{\prime})}\leq \frac{C}{|\lambda|},
				$  for  $\lambda\in\Sigma_{\theta}$.
				\item [{\rm(iv)}]
				$
				\|(\lambda+A(t))^{-1}\|_{L(H,V)}\leq \frac{C}{|\lambda|^{1/2}},
				$\text{ for } $\lambda\in\Sigma_{\theta}$.
			\end{enumerate}
		\end{proposition}
		\begin{proof}
			The proofs are standard and can be found, for example, in \cite[Proposition 6]{Haak} or \cite[Theorem 1.54]{Ouhabaz}, we omit them here.
		\end{proof}
		
		It is worth mentioning that the  $\|(\lambda+\mathcal{A}(t))^{-1}\|_{L(V^{\prime})}\leq \frac{C}{1+|\lambda|}$ for $\lambda\in\Sigma_{\theta} \cup\{0\}$ was showed in \cite[Proposition 2.1]{Arendt1} for example, which implies the boundedness of $\mathcal{A}(t)^{-1}$. For convenience, we use this expression to establish our theorem.
		
	Let $\gamma\in(0,1)$, for a sectorial operator $A$ in Hilbert space $H$, we next introduce the interpolation space 
		$$
		D_{A}(\gamma, p):=\left\{x \in H:~~ [x]_{D_{A}(\gamma, p)}<\infty\right\} ,
		$$
		where
		$$
		[x]_{D_{A}(\gamma, p)}:=\left\{\int_{0}^{\infty}\left(t^\gamma\left\|A(t I+A^{-1} x\right\|_H\right)^p \frac{d t}{t}\right\}^{\frac{1}{p}},
		$$
  endowed with the norm $\|x\|_{D_{A}(\gamma, p)}:=\|x\|_H+[x]_{D_{A}(\gamma, p)}$. It is well known that these spaces are equivalent (up to the equivalence of norms) with the interpolation spaces $(H, D(A))_{\gamma, p}$ between $H$ and $D(A)$, see e.g. \cite[Proposition 2.2.6]{Lunardi}. 
		
		Since $A$ is a sectorial operator, for $0<v<1$, $A^{-v}$ is well-defined in $L(H)$ and is given by
		$$
		A^{-v}:=\frac{\sin \pi v}{\pi} \int_0^\infty s^{-v}(s I+A)^{-1} d s.
		$$
		Let the fractional power operator $A^{v}$ be defined as the inverse of  $A^{-v}$. The domain of the fractional power operator, denoted by $D(A^v)$, is equipped with the graph norm $\|x\|_{D(A^v)}=\|x\|_H+\|A^vx\|$.  Furthermore, it is noted that $D(A^{v})$ is equivalent to the interpolation space $(H, D(A))_{v,2}$ with the equivalent norms. For details, see \cite[Chapter 2]{Yagi} for example. 
		
		Let us define two families of operators $\{\phi_{\mathcal{A}(t)}(s)\}_{t\in[0,\tau],s>0}$ and $\{\psi_{\mathcal{A}(t)}(s)\}_{t\in[0,\tau],s>0}$ using functional calculus, namely
		$$
		\begin{aligned}
			& \phi_{\mathcal{A}(t)}(s)=\frac{1}{2 \pi i} \int_{\Gamma} \lambda^{\alpha-1}e^{\lambda s}(\lambda^{\alpha} + \mathcal{A}(t))^{-1} d \lambda, \\
			& \psi_{\mathcal{A}(t)}(s)=\frac{1}{2 \pi i} \int_{\Gamma} e^{\lambda s}(\lambda^{\alpha} + \mathcal{A}(t))^{-1} d \lambda, \\
		\end{aligned}
		$$
		where
		$$ 
		\Gamma:=\left\{r e^{-i\theta} ; R \leq r<\infty\right\} \cup\left\{R e^{i \varphi} ;|\varphi| \leq \theta\right\} \cup\left\{r e^{i\theta} ; R \leq r<\infty\right\}
		$$
		is oriented counterclockwise and $R>0$, $\theta$ is the same as in Proposition \ref{proposition2.1}. 
		\begin{Lemma}\label{Lemma2.1}
			The operator families \(\{\phi_{\mathcal{A}(t)}(s)\}_{t \in [0,\tau], s > 0}\) and \(\{\psi_{\mathcal{A}(t)}(s)\}_{t \in [0,\tau], s > 0}\) are well-defined and satisfy the following properties:
			\begin{enumerate} 
				\item [{\rm (i)}] $
				\|\mathcal{A}(t)\phi_{\mathcal{A}(t)}(s)\|_{L(V')} \leq Cs^{-\alpha},$
				and $
				\|\mathcal{A}(t)\psi_{\mathcal{A}(t)}(s)\|_{L(V')} \leq Cs^{-1}.$
				\item [{\rm (ii)}] For fixed \(t \in [0,\tau]\), there hold
				\[
				\mathscr{L}(\phi_{\mathcal{A}(t)}(s)) = \lambda^{\alpha-1}(\lambda^{\alpha}+\mathcal{A}(t))^{-1},
				\quad
				\mathscr{L}(\psi_{\mathcal{A}(t)}(s)) = (\lambda^{\alpha}+\mathcal{A}(t))^{-1},
				\]
				where  $\mathscr{L}$ is the Bochner integral operator $$\mathscr{L}(f)(\lambda):=\int_0^\infty e^{-\lambda s}f(s)ds,~~Re(\lambda)>0.$$
				\item [{\rm (iii)}] \(\phi_{\mathcal{A}(t)}(s) \in H^{1}(0,t;\mathscr{L}(V'))\), and
				\[
				\frac{d}{ds} \phi_{\mathcal{A}(t)}(s) = -\mathcal{A}(t)\psi_{\mathcal{A}(t)}(s),
				\quad
				\int_{0}^{s} k(s-r)\psi_{\mathcal{A}(t)}(r)\,dr = \phi_{\mathcal{A}(t)}(s),
				\]
				where the derivative is interpreted as the distributional sense.
			\end{enumerate}
		\end{Lemma}
		\begin{proof}
			Note that for $\lambda\in \Gamma$, we have $\lambda^{\alpha}=|\lambda|^{\alpha}e^{i\alpha\theta}\in \Sigma_\theta$. From Proposition \ref{proposition2.1}, we know that $\|(\lambda^{\alpha}+\mathcal{A}(t))^{-1}\|_{L(V^{\prime})}\leq C/|\lambda^{\alpha}|$ for $\lambda\in \Gamma$, which implies
			\[
			\|\lambda^{\alpha-1}(\lambda^{\alpha}+\mathcal{A}(t))^{-1}\|_{L(V')} \leq \frac{C}{|\lambda|}, \quad 
			\|(\lambda^{\alpha}+\mathcal{A}(t))^{-1}\|_{L(V')} \leq \frac{C}{|\lambda|^{\alpha}}  \text{ on } \Gamma.
			\]
			It follows that
			\begin{equation}\label{EQ2.1}
				\left\|\phi_{\mathcal{A}(t)}(s)\right\|_{L(V^{\prime})} 
				\leq \frac{C}{2 \pi} \int_{\Gamma} e^{\operatorname{Re}(\lambda s)} \frac{|d \lambda|}{{|\lambda|}} \leq C\left(\int_R^{\infty} e^{-srcos\theta} \frac{d r}{r}+R\int_0^\pi e^{Rs\cos \varphi} d \varphi\right) \leq M,
			\end{equation}
			and similarly, choosing $R=s^{\alpha-1}$, we have
			\begin{equation}\label{EQ2.2}
				\begin{aligned}
					\left\|\psi_{\mathcal{A}(t)}(s)\right\|_{L(V^{\prime})} 
					&\leq \frac{C}{2 \pi} \int_{\Gamma} e^{\operatorname{Re}(\lambda s)} \frac{|d \lambda|}{{|\lambda|^{\alpha}}} 
					\\&\leq C\left(\int_1^{\infty} e^{-sr\cos\theta} \frac{d r}{r^{\alpha}} + s^{\alpha-1}\int_0^\pi e^{s^{\alpha}\cos \varphi} d \varphi\right) 
					\\&\leq M^{\prime}s^{\alpha-1}.
				\end{aligned}
			\end{equation}
			These estimates shows that the integral representation of $\phi_{\mathcal{A}(t)}(s)$ and $\psi_{\mathcal{A}(t)}(s)$ are absolutely convergent and therefore well-defined. Following the identity
			$$
			\mathcal{A}(t)(\lambda^{\alpha} + \mathcal{A}(t))^{-1}=I-\lambda^{\alpha}(\lambda^{\alpha} + \mathcal{A}(t))^{-1},
			$$
			we find a constant $C$ such that
			$$
			\|\mathcal{A}(t)(\lambda^{\alpha} + \mathcal{A}(t))^{-1}\|_{L(V^{\prime})}\leq1+\|\lambda^{\alpha}(\lambda^{\alpha} + \mathcal{A}(t))^{-1}\|_{L(V^{\prime})}\leq C.
			$$
			Thus
			$$
			\begin{aligned}
				\|\mathcal{A}(t)\phi_{\mathcal{A}(t)}(s)\|_{L(V^{\prime})}&=\|\frac{1}{2\pi i}\int_{\Gamma}\lambda^{\alpha-1}e^{\lambda s}\mathcal{A}(t)(\lambda^{\alpha}+\mathcal{A}(t))^{-1}d\lambda\|_{L(V^{\prime})}\\
				&\leq C\int_{\Gamma}|\lambda|^{\alpha-1}e^{\operatorname{Re}\lambda s}|d\lambda|\leq Cs^{-\alpha}.
			\end{aligned}
			$$
			The second estimate of $\mathcal{A}(t)\psi_{\mathcal{A}(t)}(s)$ in $L(V^{\prime})$ can be derived in a similar manner. Thus, conclusion (i) follows.
			
			To prove (ii), fix $\lambda\in\mathbb C$ such that $Re(\lambda)>R$. We obtain
			$$
			\begin{aligned}
				\mathscr{L}(\phi_{\mathcal{A}(t)}(s))=\int_0^{\infty} e^{-\lambda s} \phi_{\mathcal{A}(t)}(s) d s&=\frac{1}{2 \pi i} \int_{\Gamma} \mu^{\alpha-1} (\mu^{\alpha}+\mathcal{A}(t))^{-1} \int_0^{\infty} e^{-(\lambda-\mu) t} d t d \mu
				\\&= \frac{-1}{2 \pi i} \int_{\Gamma} \frac{\mu^{\alpha-1} (\mu^{\alpha}+\mathcal{A}(t))^{-1} x}{\mu-\lambda} d \mu
				\\&=\lambda^{\alpha-1} (\lambda^{\alpha}+\mathcal{A}(t))^{-1},
			\end{aligned}
			$$ 
			where we  use Fubini's theorem and Cauchy's integral formula. Therefore, it yields
			$$
			\mathscr{L}(\phi_{\mathcal{A}(t)}(s))=(\lambda^{\alpha}+\mathcal{A}(t))^{-1}.
			$$
			Finally we give the proof of (iii).
			\noindent
			Observe that
			$$
			\mathscr{L}(\int_{0}^{s}k(s-r)\psi_{\mathcal{A}(t)}(r)dr)=\mathscr{L}(k(s))\mathscr{L}(\psi_{\mathcal{A}(t)}(s))=\lambda^{\alpha-1}(\lambda^{\alpha}+\mathcal{A}(t))^{-1}=\mathscr{L}(\phi_{\mathcal{A}(t)}(s)),
			$$
			then by the uniqueness theorem for $\mathscr{L}$, see \cite[Theorem 1.7.3]{Arendt2}, we get 
			$$
			\int_{0}^{s}k(s-r)\psi_{\mathcal{A}(t)}(r)dr=\phi_{\mathcal{A}(t)}(s).
			$$
			On the other hand, using Lebesgue's dominated convergence theorem, we have
			$$
			\begin{aligned}
				\frac{d}{ds}\phi_{\mathcal{A}(t)}(s)&=\frac{d}{ds}\frac{1}{2 \pi i} \int_{\Gamma} \lambda^{\alpha-1}e^{\lambda s}(\lambda^{\alpha} + \mathcal{A}(t))^{-1} d \lambda
				=\frac{1}{2 \pi i} \int_{\Gamma} \frac{d}{ds}\lambda^{\alpha-1}e^{\lambda s}(\lambda^{\alpha} + \mathcal{A}(t))^{-1} d \lambda
				\\&=\frac{1}{2 \pi i} \int_{\Gamma} e^{\lambda s}\lambda^{\alpha}(\lambda^{\alpha} + \mathcal{A}(t))^{-1} d \lambda=-\frac{1}{2 \pi i} \int_{\Gamma} e^{\lambda s}\mathcal{A}(t)(\lambda^{\alpha} + \mathcal{A}(t))^{-1} d \lambda.
			\end{aligned}
			$$
			Therefore, from the closeness of $\mathcal{A}(t)$, we obtain
			$$
			\frac{d}{ds}\phi_{\mathcal{A}(t)}(s)=-\mathcal{A}(t)\psi_{\mathcal{A}(t)}(s).
			$$
			The proof is completed.
		\end{proof}
		\begin{proposition}
			The operator $\phi_{\mathcal{A}(t)}(s)$ converges to identity operator $1$ strongly on $V^{\prime}$ as $s \rightarrow 0$.
		\end{proposition}
		\begin{proof}
			First, let $x \in \mathcal{D}(\mathcal{A}(t))=V$. It follows that
			$$
			\begin{aligned}
				\phi_{\mathcal{A}(t)}(s)x-x & =\frac{1}{2 \pi i} \int_{\Gamma} e^{\lambda s} \lambda^{\alpha-1}\left(\lambda^\alpha I+\mathcal{A}(t)\right)^{-1} x d \lambda-x \\
				& =\frac{1}{2 \pi i} \int_{\Gamma} e^{\lambda s}\left[\lambda^{-1}\lambda^{\alpha}\left(\lambda^\alpha I+\mathcal{A}(t)\right)^{-1}-\lambda^{-1}\right] x d \lambda \\
				& =\frac{1}{2 \pi i} \int_{\Gamma} \lambda^{-1} e^{\lambda s}\left(\lambda^\alpha I+\mathcal{A}(t)\right)^{-1} d \lambda(-\mathcal{A}(t)) x .
			\end{aligned}
			$$
			Similarly, we find 
			$$
			\|\lambda^{-1}(\lambda^{\alpha}+\mathcal{A}(t))^{-1}\|_{L(V^{\prime})}\leq \frac{C}{|\lambda|^{1+\alpha}},~~~ \text{ on }~~ \Gamma.
			$$
			Choosing $R=s^{\alpha}$ in $\Gamma$, we obtain
			$$
			\begin{aligned}
				\|\phi_{\mathcal{A}(t)}(s)x-x\|_{V^{\prime}}&\leq \frac{C}{2 \pi } \int_{\Gamma}  e^{\operatorname{Re}\lambda s}\frac{|d\lambda|}{|\lambda|^{1+\alpha}} \|x\|_V 
				\\&\leq \frac{C}{2 \pi } \left(\int_{1}^{\infty}  e^{-rscos\theta}\frac{dr}{r^{1+\alpha}}+s^{\alpha}\int_{0}^{\pi}  e^{-Rscos\varphi}d\varphi\right)\|x\|_V
				\\&\leq Cs^{\alpha}\|x\|_V.
			\end{aligned}
			$$
			We conclude that $\|\phi_{\mathcal{A}(t)}(s)x-x\|_{V^{\prime}}\rightarrow 0$ as $s\rightarrow0$.  Next, let $x \in V^{\prime}$ be a general vector. The denseness of $V$ in $V^{\prime}$ and uniform boundedness of the norms $\|\phi_{\mathcal{A}(t)}(s)\|_{L(V^{\prime})}$, readily implies that $\phi_{\mathcal{A}(t)}(s)x \rightarrow x$ in $V^{\prime}$.
		\end{proof}
		\begin{remark}\label{remark1}
			In view of this convergence, we are naturally led to define the value of $\phi_{\mathcal{A}(t)}(s)x$ for $x\in V^{\prime}$ at
			$s = 0$ by $\phi_{\mathcal{A}(t)}(s)x=x$.
		\end{remark}
		Next, we present a representation formula for the solution $u$ of problem (FP) in $V^{\prime}$.  Fix $f\in C_c(0,\infty;H)$ and $u_0\in H$. Recall that, according to Zacher's work \cite{Zacher} (also see Achache \cite{Achache3}),  there exist $u\in L_2(0,\tau;V)$ and $v \in H^1(0,\tau; V^{\prime})$ with $v(0)=0$, satisfying $$v(t)=\int_{0}^{t}k(t-s)(u(s)-u_0)ds.$$
		\begin{Lemma}
			For almost every $t\in(0,\tau)$, the identity in $V^{\prime}$ holds
			$$
			\begin{aligned}
				u(t)=&\phi_{\mathcal{\mathcal{A}}(t)}(t)u_0+\int_{0}^{t}\psi_{\mathcal{A}(t)}(t-s)f(s)ds
				\\&+\int_{0}^{t}\psi_{\mathcal{A}(t)}(t-s)(\mathcal{A}(t)-\mathcal{A}(s))u(s)ds.
			\end{aligned}
			$$
		\end{Lemma}
		\begin{proof}
			Since $v(\cdot)\in H^1(0,\tau; V^{\prime})$, we take a distributional derivative 
			$$
			(\phi_{\mathcal{A}(t)}(t-s)v(s))'=\psi_{\mathcal{A}(t)}(t-s)\mathcal{A}(t)v(s)+\phi_{\mathcal{A}(t)}(t-s)(-\mathcal{A}(s)u(s)+f(s)).
			$$
			Integrating both sides of the equation with respected to $s$,   the left-hand side becomes
			$$
			\phi_{\mathcal{A}(t)}(t-s)v(s)|^t_0=\phi_{\mathcal{A}(t)}(0)v(t)-\phi_{\mathcal{A}(t)}(t)v(0)=v(t)=\int_{0}^{t}k(t-s)(u(s)-u_0),
			$$
			where $\phi_{\mathcal{A}(t)}(0)v(t)=v(t)$ by Remark \ref{remark1}. For the right-hand side, 
			observe that
			$$
			\begin{aligned}
				\int_{0}^{t}\psi_{\mathcal{A}(t)}(t-s)v(s)ds
				&=\int_{s=0}^{t}\psi_{\mathcal{A}(t)}(t-s)\int_{r=0}^{s}k(s-r)(u(r)-u_0)ds
				\\&=\int\int_{0\leq r\leq s\leq t}\psi_{\mathcal{A}(t)}(t-s)k(s-r)(u(r)-u_0)drds
				\\&=\int_{r=0}^{t}\int_{t}^{s=r}\psi_{\mathcal{A}(t)}(t-s)k(s-r)(u(r)-u_0)dsdr
				\\&=\int_{r=0}^{t}\left(\int_{s=0}^{t-r}\psi_{\mathcal{A}(t)}(t-r-s)k(s)ds\right)(u(r)-u_0)dr
				\\&=\int_{r=0}^{t}\phi_{\mathcal{A}(t)}(t-r)(u(r)-u_0)dr.
			\end{aligned}
			$$
			Indeed this is the associativity of convolution. By a similar argument,
			we rewrite the right-hand side  as
			$$
			\begin{aligned}
				\int_{0}^{t}\phi_{\mathcal{A}(t)}(t-s)(\mathcal{A}(t)-\mathcal{A}(s))u(s)ds&-\int_{0}^{t}\mathcal{A}(t)\phi_{\mathcal{A}(t)}(t-s)u_0ds
				\\&+\int_{0}^{t}\phi_{\mathcal{A}(t)}(t-s)f(s)ds=\int_{0}^{t}k(t-r)x(r)dr,
			\end{aligned}
			$$
			where
			$$
			\begin{aligned}
				x(r)=\int_{0}^{r}\psi_{\mathcal{A}(t)}(r-s)(\mathcal{A}(t)-\mathcal{A}(s))u(s)ds&-\int_{0}^{r}\mathcal{A}(t)\psi_{\mathcal{A}(t)}(r-s)u_0ds
				\\&+\int_{0}^{r}\psi_{\mathcal{A}(t)}(r-s)f(s)ds.
			\end{aligned}
			$$
			Since
			$$
			\int_{s=0}^{r}\mathcal{A}(t)\psi_{\mathcal{A}(t)}(r-s)u_0ds=-\phi_{\mathcal{A}(t)}(r-s)u_0|_0^r=\phi_{\mathcal{A}(t)}(r)u_0-u_0,
			$$
			and  $u\in L_2(0,\tau;V)$,
			with the estimate of  $\phi_{\mathcal{A}(t)}$ and $\psi_{\mathcal{A}(t)}$ in \eqref{EQ2.1} and \eqref{EQ2.2},  it is clear that  
			$$
			x(r)\in L_2(0,\tau;V^{\prime}).
			$$ 
			Then we have
			$$
			\int_0^t k(t-s)(x(s)-(u(s)-u_0))ds=0,~~ \text{ for all } t\in(0,\tau),
			$$
			which implies that
			$$
			x(t)=u(t)-u_0 \text{ a.e. on }  (0,\tau).
			$$
			Therefore we get in $V^{\prime}$,
			$$
			\begin{aligned}
				u(t)=&\phi_{\mathcal{\mathcal{A}}(t)}(t)u_0+\int_{0}^{t}\psi_{\mathcal{A}(t)}(t-s)f(s)ds
				\\&+\int_{0}^{t}\psi_{\mathcal{A}(t)}(t-s)(\mathcal{A}(t)-\mathcal{A}(s))u(s)ds,\ \ \text{ a.e. on }  (0,\tau).
			\end{aligned}
			$$
		\end{proof}
		
		\section{The proofs of maximal $L_p$-regularity} \label{sect:4}
		In this section, we will provide the complete proof at the end.
		Introducing the transformations for a sufficiently large $\delta>0$: 
		$u^{\delta}(t)=u(t)e^{-\delta t}$, $\phi^{\delta}_{\mathcal{A}(t)}(s)=\phi_{\mathcal{A}(t)}(s)e^{-\delta s}$, $\psi^{\delta}_{\mathcal{A}(t)}(s)=\psi_{\mathcal{A}(t)}(s)e^{-\delta s}$, $f^{\delta}(t)=f(t)e^{-\delta t}$, we first verify several lemmas for the abstract Volterra equation  \eqref{EQ4.1}.

		\begin{Lemma}\label{4.1}
			The following decomposition is well-defined
			\begin{equation*}
				\mathcal{A}(t)u^{\delta}(t)=(Q(\mathcal{A}(\cdot)u^{\delta}(\cdot))(t)+(Lf)(t)+(Ru_0)(t),
			\end{equation*}
			where 
			\begin{align*}
				(Qg)(t):=&\int_{0}^{t}\mathcal{A}(t)\psi^{\delta}_{\mathcal{A}(t)}(t-s)(\mathcal{A}(t)-\mathcal{A}(s))\mathcal{A}(s)^{-1}g(s)ds,\\
				(Lf)(t)=&\mathcal{A}(t)\int_{0}^{t}\psi^{\delta}_{\mathcal{A}(t)}(t-s)f(s)ds,
			\end{align*}
			and $(Ru_0)(t)=\mathcal{A}(t)\phi^{\delta}_{\mathcal{A}(t)}(t)u_0.
			$
		\end{Lemma}
		\begin{proof}
			It suffices to show that each term in the \eqref{EQ4.1} is in $V=D(\mathcal{A}(t))$.  For $u \in V$, using (i) in Lemma \ref{Lemma2.1},  we have
			$$
			\begin{aligned}
				\|\mathcal{A}(t)\psi^{\delta}_{\mathcal{A}(t)}(t-s)(\mathcal{A}(t)-\mathcal{A}(s))u^{\delta}(s)\|_{V^{\prime}}
				&\leq\frac{C}{t-s}\|(\mathcal{A}(t)-\mathcal{A}(s))u(s)\|_{V^{\prime}}
				\\&=\frac{C}{t-s}\sup_{\|v\|_V=1}|a(t;u(s),v)-a(s;u(s),v)|
				\\&\leq C\frac{\omega(t-s)}{t-s}\|u(s)\|_{V}.
			\end{aligned}
			$$
			Recall that $\|u(\cdot)\|_{V}\in L_2(0,\tau;\mathbb R)$ and $\omega(t)/t\in L_1(0,\tau;\mathbb R)$ by virtue of \eqref{A1}. Therefore,  Young's convolution inequality shows that
			$$
			\mathbf{1}_{[0,t]}(s)\mathcal{A}(t)\psi^{\delta}_{\mathcal{A}(t)}(t-s)(\mathcal{A}(t)-\mathcal{A}(s))u^{\delta}(s)ds\in L_1(0,\tau;V^{\prime}),
			$$
			where $\mathbf{1}_{[0,t]}$ denotes the characteristic function on $[0,t]$, which also implies 
			$$			\int_{0}^{t}\psi^{\delta}_{\mathcal{A}(t)}(t-s)(\mathcal{A}(t)-\mathcal{A}(s))u^{\delta}(s)ds \in V.
			$$
			In addition,
			$$
			\|(Ru_0)(t)\|_{V^{\prime}}=\|\mathcal{A}(t)\phi^{\delta}_{\mathcal{A}(t)}(t)u_0\|_{V^{\prime}}\leq Ct^{-\alpha}\|u_0\|_{V^{\prime}},
			$$
			which implies $\phi^{\delta}_{\mathcal{A}(t)}(t)u_0\in D(\mathcal{A}(t))$.
			Since $u^{\delta}\in V$, by \eqref{EQ4.1} we get $$\int_{0}^{t}\psi^{\delta}_{\mathcal{A}(t)}(t-s)f(s)ds\in V,$$ immediately. The proof is completed.
		\end{proof}
		\begin{Lemma}\label{4.2}
			$(I-Q)^{-1}$ is bounded on $L_p(0, \tau ; H)$.
		\end{Lemma}
		\begin{proof}
			It suffices to show that $Q$ is bounded and contractive on $L_p(0, \tau ; H)$ by Neumann's theorem. Note that
			$$
			\|(Qg)(t)\|_H\leq\int_{0}^{t}\|\mathcal{A}(t)\psi^{\delta}_{\mathcal{A}(t)}(t-s)(\mathcal{A}(t)-\mathcal{A}(s))\mathcal{A}(s)^{-1}g(s)\|_H ds,
			$$
			it allows us to rewrite $Qg$ by the functional calculus as
			$$
			\begin{aligned}
				\mathcal{A}(t)&\psi^{\delta}_{\mathcal{A}(t)}(t-s)(\mathcal{A}(t)-\mathcal{A}(s))\mathcal{A}(s)^{-1}g(s)\\&=e^{-\delta(t-s)}\int_{\Gamma} e^{\lambda(t-s)}\lambda^{\alpha}(\lambda^{\alpha}+\mathcal{A}(t))^{-1}(\mathcal{A}(t)-\mathcal{A}(s))\mathcal{A}(s)^{-1}g(s)d\lambda.
			\end{aligned}
			$$
			For any $u, v\in H$, we observe that
			\begin{equation}\label{EQA}
				\begin{aligned}
					&((\lambda^{\alpha}+\mathcal{A}(t))^{-1}(\mathcal{A}(t)-\mathcal{A}(s))\mathcal{A}(s)^{-1}u,v )_H \\&=((\mathcal{A}(t)-\mathcal{A}(s))\mathcal{A}(s)^{-1}u,(\overline{\lambda^{\alpha}}+\mathcal{A}^{*}(t))^{-1}v )_H
					\\&=(\mathcal{A}(t)\mathcal{A}(s)^{-1}u,(\overline{\lambda^{\alpha}}+\mathcal{A}^{*}(t))^{-1}v )_H-(\mathcal{A}(s)\mathcal{A}(s)^{-1}u,(\overline{\lambda^{\alpha}}+\mathcal{A}^{*}(t))^{-1}v )_H
					\\&=a(t;\mathcal{A}(s)^{-1}u,(\overline{\lambda^{\alpha}}+\mathcal{A}^{*}(t))^{-1}v)-a(s;\mathcal{A}(s)^{-1}u,(\overline{\lambda^{\alpha}}+\mathcal{A}^{*}(t))^{-1}v)
					\\&\leq C\omega(t-s)\|\mathcal{A}(s)^{-1}u\|_V\|\overline{\lambda^{\alpha}}+\mathcal{A}^{*}(t))^{-1}v\|_V.
				\end{aligned}
			\end{equation}
			By proposition \ref{proposition2.1} and the boundedness of $A(s)^{-1}$, we have
			$$
			\|\mathcal{A}(s)^{-1}u\|_V\leq C\|u\|_H, \ \ \text{and} \ \ \|(\overline{\lambda^{\alpha}}+\mathcal{A}^{*}(t))^{-1}v\|_V\leq \frac{C}{|\lambda|^{\frac{\alpha}{2}}}\|v\|_H.
			$$
			Therefore, we get
			$$
			\begin{aligned}
				&\left\|\int_{\Gamma} e^{\lambda(t-s)}\lambda^{\alpha}(\lambda^{\alpha}+\mathcal{A}(t))^{-1}(\mathcal{A}(t)-\mathcal{A}(s))\mathcal{A}(s)^{-1}g(s)d\lambda\right\|_H
				\\&\leq C \int_{\Gamma}|\lambda|^{\frac{\alpha}{2}} e^{-(t-s)\operatorname{Re}\lambda}\omega(t-s)\|g(s)\|_H |d\lambda|
				\\&\leq C\frac{\omega(t-s)}{(t-s)^{1+\alpha/2}}\|g(s)\|_H,
			\end{aligned}
			$$
			from which it follows that
			$$
			\|(Qg)(t)\|_H\leq C \int_{0}^{t}e^{-\delta(t-s)}\frac{\omega(t-s)}{(t-s)^{1+\alpha/2}}\|g(s)\|_H ds.
			$$
			Then, by Young's convolution inequality, (\ref{A1}) implies that $Q$ is bounded on $L_1(0, \tau ; H)$  as
			$$
			\| Qg \|_{L_1(0,\tau;H)}\leq C\|{\omega(r)}e^{-\delta r} {r^{-(1+\alpha/2)}}\|_{L_1(0,\tau;\mathbb R)}\|g\|_{L_1(0,\tau;H)},
			$$
			and it is also bounded on $L_{\infty}(0, \tau ; H)$ by 
			$$
			\|(Qg)(t)\|_H\leq C \|{\omega(r)}e^{-\delta r} {r^{-(1+\alpha/2)}}\|_{L_1(0,\tau;\mathbb R)}\|g \|_{L_{\infty}(0,\tau;H)}.
			$$
			Therefore, $Q$ is bounded on $L_p(0, \tau ; H)$ by Riesz-Thorin interpolation theorem.
			Since $\omega$ is continuous, applying the mean value theorem, we have for some $\xi\in (0,\delta)$,
			$$
			\int_{0}^{t}e^{-\delta r}\frac{\omega(r)}{r^{1+\alpha/2}}dr=e^{-\delta \xi}\int_{0}^{t}\frac{\omega(r)}{r^{1+\alpha/2}}dr.
			$$
			By choosing $\delta$ large enough, we make $Q$ strictly contractive in $L_p(0,\tau;H)$, so that $(I-Q)^{-1}$ is bounded by the Neumann series. The proof is completed. 
		\end{proof}
		
		\begin{Lemma}\label{4.3}
			The following three different cases for $\alpha\in(0,1)$ are true:
			\begin{itemize}
				\item [{\rm (i)}] For $\alpha<\frac{1}{p}$,    $Rx\in L_p(0, \tau ; H)$ if $x\in H$.
				\item [{\rm (ii)}] For $\alpha=\frac{1}{p}$,    $Rx\in L_p(0, \tau ; H)$ if $x\in (H, \mathcal{D}(A(0)))_{\epsilon, 2}$, with the associated estimate
				$$
				\|R x\|_{L_p(0, \tau ; H)} \leq C\|x\|_{(H, \mathcal{D}(A(0)))_{\epsilon, 2}}.
				$$ 
				\item [{\rm (iii)}] For $\alpha>\frac{1}{p}$,   $Rx\in L_p(0, \tau ; H)$ if $x\in(H, \mathcal{D}(A(0)))_{1-\frac{1}{\alpha p}, p}$, with the associated estimate
				$$
				\|R x\|_{L_p(0, \tau ; H)} \leq C\|x\|_{(H, \mathcal{D}(A(0))_{1-\frac{1}{\alpha p}, p}}.
				$$
			\end{itemize}
		\end{Lemma}
		\begin{proof}
			To achieve this proof, we first consider the $L_p$-norm of $(R_0 x)(t):=A(0)\phi^{\delta}_{A(0)}(t)x$ for $x\in H$. For $\alpha<\frac{1}{p}$, a similar argument to (i) in Lemma \ref{Lemma2.1} implies that
			$$\|R_0 x\|_{L_p(0, \tau ; H)} \leq C\| {t^{-\alpha}} \|_ {L_p(0, \tau ; \mathbb{R})}\|x\|_H\leq C\|x\|_H.$$
			For $\alpha=\frac{1}{p}$, let $x\in D(A(0)^{\epsilon})=(H, \mathcal{D}(A(0)))_{\epsilon, 2}$ for $\epsilon\in(0,1)$. We rewrite
			$$\begin{aligned}
				( R_0 x)(t)  = A(0)^{1-\epsilon}\phi^{\delta}_{A(0)}(t)A(0)^{\epsilon}x = \int_{\Gamma}\lambda^{\alpha-1}e^{\lambda t}A(0)^{1-\epsilon}(\lambda^{\alpha}+A(t))^{-1}A(0)^{\epsilon}xd\lambda .
			\end{aligned}$$
			By virtue of the moment inequality, see e.g. \cite[Section 2.7.4]{Yagi}, for $y\in H$, we have
			$$\|A(0)^{1-\epsilon}(\lambda^{\alpha}+A(t))^{-1}y\|_H\leq \|A(0)(\lambda^{\alpha}+A(t))^{-1}y\|^{1-\epsilon}_H\|(\lambda^{\alpha}+A(t))^{-1}y\|^{\epsilon}_H\leq C\lambda^{\epsilon}.$$
			which implies that 
			$$\|(R_0 x)(t)\|_{H}\leq Ct^{\alpha(1-\epsilon)}\|x\|_{(H, \mathcal{D}(A(0)))_{\epsilon, 2}}.$$
			Therefore, it yields
			$$\begin{aligned}
				\|R_0 x\|_{L_p(0, \tau ; H)} \leq {C}\| {t^{-\alpha(1-\epsilon)}}\|_ {L_p(0, \tau ; \mathbb{R})}\|x\|_{(H, \mathcal{D}(A(0)))_{\epsilon, 2}}  \leq C\|x\|_{(H, \mathcal{D}(A(0)))_{\epsilon, 2}}.
			\end{aligned}$$
			For $\alpha>\frac{1}{p}$, choosing the contour 
			$$ 
			\Gamma'=\left\{r e^{-i\theta} ; R \leq r<\infty\right\} \cup\left\{\epsilon e^{i \varphi} ;|\varphi| \leq \theta\right\} \cup\left\{r e^{i\theta} ; R \leq r<\infty\right\},
			$$
			one can rewrite the $R_0x$ into
			$$
			\begin{aligned}
				(R_0 x)(t)
				&=\frac{e^{-\delta t}}{2\pi i}\int_{\Gamma'}A(0)\lambda^{\alpha-1}e^{\lambda t}(\lambda^{\alpha}+A(0))^{-1}xd\lambda
				\\=&\frac{e^{-\delta t}e^{i\theta}}{\pi i}\int_{\epsilon}^{\infty}A(0)(re^{i\theta})^{\alpha-1}e^{re^{i\theta} t}((re^{i\theta})^{\alpha}+A(0))^{-1}xdr
				\\&+\frac{e^{-\delta t}}{2\pi}\int_{-\theta}^{\theta}A(0)(\epsilon e^{i\varphi})^{\alpha}e^{\epsilon e^{i\varphi} t}((\epsilon e^{i\varphi})^{\alpha}+A(0))^{-1}xd\varphi.
			\end{aligned}
			$$
			Note that 
			$$
			\|\int_{-\theta}^{\theta}A(0)(\epsilon e^{i\varphi})^{\alpha}e^{\epsilon e^{i\varphi} t}((\epsilon e^{i\varphi})^{\alpha}+A(0))^{-1}xd\varphi\|_H \rightarrow 0  \text{ as  } \epsilon\rightarrow 0,
			$$
			hence, it is just sufficient to estimate
			$$
			\|\int_{0}^{\infty}A(0)(re^{i\theta})^{\alpha-1}e^{re^{i\theta} t}((re^{i\theta})^{\alpha}+A(0))^{-1}xdr\|_H\leq C\|x\|_{(H, \mathcal{D}(A(0)))_{1-\frac{1}{\alpha p}},p}.
			$$
			Observe that
			$$
			\begin{aligned}
				&((re^{i\theta})^{\alpha}+A(0))^{-1}-(r^{\alpha}+A(0))^{-1}
				\\&=((re^{i\theta})^{\alpha}-r^{\alpha})((re^{i\theta})^{\alpha}+A(0))^{-1}(r^{\alpha}+A(0))^{-1}
				\\&=(1-e^{-i\alpha\theta})(re^{i\theta})^{\alpha}((re^{i\theta})^{\alpha}+A(0))^{-1}(r^{\alpha}+A(0))^{-1},
			\end{aligned}
			$$
			which implies that
			$$
			\|A(0)((re^{i\theta})^{\alpha}+A(0))^{-1}x\|_{H}\leq (1+2C\sin\frac{\alpha\theta}{2})\|A(0)(r^{\alpha}+A(0))^{-1}x\|_{H}.
			$$
			Therefore we have
			$$
			\begin{aligned}
				\|(R_0 x)(t)\|_{H}\leq &C\int_{0}^{\infty}e^{-rt\cos\theta }\|r^{\alpha-1}A(0)(r^{\alpha}+A(0))^{-1}x\|_{H}dr.
			\end{aligned}
			$$
			Applying a change of variable $\rho=rt$ and the generalized Minkowski inequality 
			$$
			\left\{\int_{\mathbb{R}_{+}} d \tau\left|\int_{\mathbb{R}_{+}} f(\rho, t) d \rho\right|^p\right\}^{1 / p} \leq \int_{\mathbb{R}_{+}} d \rho\left\{\int_{\mathbb{R}_{+}}|f(\rho, t)|^p d t\right\}^{1 / p},
			$$
			we deduce that
			$$
			\begin{aligned}
				\| R_0 x \|_{L_p(0,\tau;H)}
				&\leq C\int_{0  }^{\infty}e^{-\rho \cos \theta}\left(\int_{0}^{\tau}\Big(\|\frac{\rho^{\alpha-1}}{t^{\alpha-1}}A(0)(\frac{\rho^{\alpha}}{t^{\alpha}}+A(0))^{-1}x\|_{H}\frac{1}{t}\Big)^p dt\right)^{\frac{1}{p}}d\rho
				\\&\leq C\int_{0}^{\infty}e^{-\rho \cos \theta}\rho^{\frac{1}{p}-1}d\rho\left(\int_{0}^{\infty}(\|\sigma^{\alpha-\frac{1}{p}}A(0)(\sigma^{\alpha}+A(0))^{-1}x\|_{H})^p \frac{d\sigma}{\sigma}\right)^{\frac{1}{p}}
				\\&=C[x]_{D_{A(0)}(1-\frac{1}{\alpha p},p)}\leq C\|x\|_{(H, \mathcal{D}(A(0)))_{1-\frac{1}{\alpha p}, p}}.
			\end{aligned}
			$$
			The inequality in second line is due to the change of variable $\sigma=\rho/t$, the third line is due to the fact
			$$
			\begin{aligned}
				\relax[x]_{D_{A(0)}(\gamma, p)}=&\left\{\int_{0}^{\infty}\left(s^\gamma\left\|A(0)(s I+A(0))^{-1} x\right\|_{H}\right)^p \frac{d s}{s}\right\}^{\frac{1}{p}}
				\\=&\left\{\int_{0}^{\infty}\left(\sigma^{\alpha\gamma}\left\|A(0)(\sigma^{\alpha} I+A(0))^{-1} x\right\|_{H}\right)^p \frac{d \sigma}{\sigma}\right\}^{\frac{1}{p}}.
			\end{aligned}
			$$
			In the sequel, we need to consider the estimate of $(R-R_0)x$. Note that
			$$
			\begin{aligned}
				&A(t)(\lambda^{\alpha}+A(t))^{-1}-
				A(0)(\lambda^{\alpha}+A(0))^{-1}
				\\&=\lambda^{\alpha}(\lambda^{\alpha}+A(t))^{-1}(A(t)-A(0))(\lambda^{\alpha}+A(0))^{-1}.
			\end{aligned}
			$$
			Using the same arguments as in \eqref{EQA}, we get
			$$
			\|\lambda^{\alpha}(\lambda^{\alpha}+A(t))^{-1}(A(t)-A(0))(\lambda^{\alpha}+A(0))^{-1}\|_{L(H)}\leq C\omega(t).
			$$
			Therefore, we have
			$$
			\begin{aligned}
				\|(R-R_0)(t)x\|_H&=\|\int_{\Gamma}\lambda^{\alpha-1}e^{\lambda t}\lambda^{\alpha}(\lambda^{\alpha}-A(t))^{-1}(A(t)-A(0))(\lambda^{\alpha}-A(0))^{-1}xd\lambda\|_H
				\\&\leq \int_{\Gamma}\omega(t)|\lambda^{\alpha-1}|e^{Re{\lambda}t}|d\lambda| \|x\|_{H}\leq C\frac{\omega(t)}{t^{\alpha}}\|x\|_H.
			\end{aligned}
			$$
			If $\alpha\leq \frac{1}{p}$, it is straightforward to obtain $(R-R_0)x\in L_p(0,\tau; H)$, while if $\alpha>\frac{1}{p}$, it follows from (\ref{A2}) that $(R-R_0)x\in L_p(0,\tau; H)$.
			The proof is completed.
		\end{proof}

		In the following, we consider the $L_p$-boundedness of the operator
		$$
		(Lf)(t)=\mathcal{A}(t)\int_{0}^{t}\psi^{\delta}_{\mathcal{A}(t)}(t-s)f(s)ds.
		$$
		Our strategy is to transform the operator $L$ into a pseudo-differential operator. Let $f \in C_c^{\infty}(0, \tau; H)$, and extend the domain of $f$ to all of $\mathbb{R}$ by defining it as zero outside the interval $(0, \tau)$. This extended function, denoted by $f_0$, belongs to the Schwarz class $\mathcal{S}(\mathbb{R}; H)$. We use $\mathcal{F}f_0$ or $\widehat{f_0}$ to represent its Fourier transform of $f_0$. It is clear that
		$$
		\begin{aligned}
			\int_{0}^{t}\psi^{\delta}_{\mathcal{A}(t)}(t-s)f(s)ds&=\int_{-\infty}^{t}e^{-\delta (t-s)}\psi_{\mathcal{A}(t)}(t-s)\int_{\mathbb{{\mathbb{R}}}}e^{is\xi}\widehat{f_0}(\xi)d\xi ds
			\\&=\int_{\mathbb{R}}\Big(\int_{-\infty}^{t}e^{-(i\xi+\delta)(t-s)}\psi_{\mathcal{\mathcal{A}}(t)}(t-s)ds\Big)\widehat{f_0}(\xi)e^{it\xi}d\xi.
		\end{aligned}
		$$
		By virtue of (ii) in Lemma \ref{Lemma2.1} we have
		$$
		\int_{-\infty}^{t}e^{-(i\xi+\delta)(t-s)}\psi_{\mathcal{\mathcal{A}}(t)}(t-s)ds=((i\xi+\delta)^{\alpha}+\mathcal{A}(t))^{-1},
		$$
		which implies that
		$$
		Lf(t)=\int_{\mathbb{R}} \mathcal{A}(t)((i\xi+\delta)^{\alpha}+\mathcal{A}(t))^{-1}\widehat{f_0}(\xi)e^{it\xi}d\xi=\mathcal{F}^{-1}(\sigma(t,\xi)\widehat{f_0}(\xi))(t),
		$$
		where
		$$
		\sigma(t,\xi)= \begin{cases}\mathcal{A}(0)((i\xi+\delta)^{\alpha}+\mathcal{A}(0))^{-1}, & \text { if } t<0, \\ \mathcal{A}(t)((i\xi+\delta)^{\alpha}+\mathcal{A}(t))^{-1}, & \text { if } 0 \leq t \leq \tau, \\ \mathcal{A}(\tau)((i\xi+\delta)^{\alpha}+\mathcal{A}(\tau))^{-1}, & \text { if } t>\tau.
		\end{cases}
		$$
		We introduce the following lemma about the $L_2$-boundedness of pseudo-differential operators, its proof can be found in Haak and Ouhabaz \cite{Haak}.
		\begin{Lemma} \label{4.4}
			Let $T f(x):=\frac{1}{(2 \pi)^n} \int_{\mathbb{R}^n} \sigma(x, \xi) \widehat{f}(\xi) e^{i x \xi} \mathrm{~d} \xi$ for $f\in\mathcal{S}\left(\mathbb{R}^n ; H\right)$. Suppose there exists a non-decreasing function $\omega:[0, \infty) \rightarrow$ $[0, \infty)$ such that for all $|\beta| \leq[n / 2]+2$, the following conditions hold:
			$$
			\left\|\partial_{\xi}^\beta \sigma(x, \xi)\right\|_{L(H)} \leq \frac{C_\beta}{(1+\xi^2)^{\beta/2}},
			$$
			and
			$$
			\left\|\partial_{\xi}^\beta \sigma(x, \xi)-\partial_{\xi}^\beta \sigma\left(x^{\prime}, \xi\right)\right\|_{L(H)} \leq \frac{C_\beta}{(1+\xi^2)^{\beta/2}} \omega\left(\left|x-x^{\prime}\right|\right),
			$$
			for some positive constant $C_\beta$. In addition, suppose that
			$$
			\int_0^1 \frac{(\omega(t))^2 }{t}\mathrm{d} t<\infty,
			$$
			then $T$ is a bounded operator on $L_2\left(\mathbb{R}^n ; H\right)$.
		\end{Lemma}
		\begin{Lemma}\label{4.5}
			Operator $L$ is bounded on $L_2(0, \tau ; H)$.
		\end{Lemma}
		\begin{proof}
			Since $(i\xi+\delta)^{\alpha}\in \rho(-\mathcal{A}(t))$, there exists a constant $M>0$ such that
			$$
			\|\sigma(x, \xi)\|_{L(H)} \leq M.
			$$
			By using Cauchy's formula for derivative of holomorphic function, we get
			$$
			\|\partial_{\xi}^k\sigma(x, \xi)\|_{L(H)}\leq \frac{C_{\theta,k}}{(1+\xi^2)^{k/2}}.
			$$
			Recall that
			$$
			\begin{aligned}
				&\mathcal{A}(t)(\lambda^{\alpha}+\mathcal{A}(t))^{-1}-\mathcal{A}(s)(\lambda^{\alpha}+\mathcal{A}(s))^{-1}
				\\&=\lambda^{\alpha}(\lambda^{\alpha}+\mathcal{A}(t))^{-1}(\mathcal{A}(t)-\mathcal{A}(0))(\lambda^{\alpha}+\mathcal{A}(s))^{-1},
			\end{aligned}
			$$
			and for $\lambda\in\rho(-\mathcal{A}(\cdot))$,
			$$
			\|\lambda^{\alpha}(\lambda^{\alpha}+\mathcal{A}(t))^{-1}(\mathcal{A}(t)-\mathcal{A}(0))(\lambda^{\alpha}+\mathcal{A}(s))^{-1}\|\leq C\omega(t-s) ,
			$$ 
			the following estimate holds
			$$
			\|\sigma(x, \xi)-\sigma(y, \xi)\|_{L(H)}\leq C\omega(x-y),
			$$
			and particularly
			$$
			\|\partial_{\xi}^k\sigma(x, \xi)-\partial_{\xi}^k\sigma(y, \xi)\|_{L(H)}\leq C_{\theta,k}\frac{\omega(x-y)}{(1+\xi^2)^{k/2}}.
			$$
			The proof is completed.
		\end{proof}
		\begin{Lemma}\label{Lemma4.6} 
			Operator $L$ is bounded on $L_p(0, \tau ; H)$ for all $p \in(1, \infty)$.
		\end{Lemma}
		\begin{proof}
			The operator $L$ can be regarded as a singular integral operator whose kernel is
			$$
			K(t, s)=\mathbf{1}_{\{0 \leq s \leq t \leq \tau\}} \mathcal{A}(t)\psi_{A(t)}(t-s).
			$$
			We aim to show that $L$ and $L_*$ belong to the weak (1,1) type.  Following this, the Marcinkiewicz interpolation theorem and the $L_2$-boundedness in Lemma \ref{4.5} will conclude that operator $L$ is bounded on $L_p(0, \tau ; H)$ for all $p \in(1, \infty)$. As established in the literature (see, for instance, \cite[Theorems III.1.2, III.1.3]{Francia}), both $L$ and $L_*$ are operator of weak $(1,1)$ type 
 provided that the associated kernel $K(t,s)$ satisfies the H\"{o}rmander-type integral condition
			\begin{equation}\label{4.6}
				\int_{|t-s| \geq 2|s^{\prime}-s|}\|K(t, s)-K(t, s^{\prime})\|_{L(H)} d t \leq C,
			\end{equation}
			and
			\begin{equation}\label{4.7}
				\int_{|t-s| \geq 2|s^{\prime}-s|}\|K(s, t)-K(s^{\prime}, t)\|_{L(H)} d t \leq C,
			\end{equation}
			for some constant $C>0$ independent of $s, s^{\prime} \in(0, \tau)$.
			
			Let us first consider the inequality in \eqref{4.6}. 
			By Fubini' theorem, one can write
			$$
			\begin{aligned}
				\psi_{A(t)}(t-s)-\psi_{A(t)}(t-s^{\prime})&=\int_{\Gamma}(e^{(t-s)\lambda}-e^{(t-s^{\prime})\lambda})(\lambda^{\alpha}+\mathcal{A}(t))^{-1}d\lambda
				\\&=\int_{\Gamma}\left(\int_{t-s}^{t-s^{\prime}}\lambda e^{r \lambda}dr\right)(\lambda^{\alpha}+\mathcal{A}(t))^{-1}d\lambda
				\\&=\int_{t-s}^{t-s^{\prime}}\left(\int_{\Gamma}\lambda e^{r \lambda}(\lambda^{\alpha}+\mathcal{A}(t))^{-1}d\lambda\right)dr.
			\end{aligned}
			$$
			By calculation, it yields
			$$
			\begin{aligned}
				\int_{\Gamma}\lambda e^{r \lambda}\mathcal{A}(t)(\lambda^{\alpha}+\mathcal{A}(t))^{-1}d\lambda
				\leq C\int_{\Gamma}|\lambda| e^{r \operatorname{Re}\lambda}|d\lambda|
				\leq Cr^{-2}.
			\end{aligned}
			$$
			Then we obtain
			$$
			\|\mathcal{A}(t)(\psi_{\mathcal{A}(t)}(t-s)-\psi_{\mathcal{A}(t)}(t-s^{\prime}))\|_{L(H)}\leq C\Big|\int_{t-s}^{t-s^{\prime}}r^{-2}dr\Big|= C\Big|\frac{1}{t-s^{\prime}}-\frac{1}{t-s}\Big|.
			$$
			Therefore we have
			$$
			\begin{aligned}
				& \int_{|t-s| \geq 2|s^{\prime}-s|}\|K(t, s)-K(t, s^{\prime})\|_{L(H)} \mathrm{d} t \\
				&  =\int_{|t-s| \geq 2|s^{\prime}-s|}\|\mathcal{A}(t) (\psi_{\mathcal{A}(t)}(t-s)-\psi_{\mathcal{A}(t)}(t-s^{\prime}))\|_{L(H)} \mathrm{d} t \\
				&  \leq C \int_{|t-s| \geq 2|s^{\prime}-s|}\Big|  \frac{1}{t-s^{\prime}}-\frac{1}{t-s}\Big| \mathrm{d} t\leq C\log 2.
			\end{aligned}
			$$
			We now consider \eqref{4.7}. We observe that
			$$
			\begin{aligned}
				& \int_{|t-s| \geq 2|s^{\prime}-s|}\|K(s, t)-K(s^{\prime}, t)\|_{L(H)} \mathrm{d} t \\
				& =\int_{|t-s| \geq 2|s^{\prime}-s|}\|\mathcal{A}(s) \psi_{\mathcal{A}(s)}(s-t)-(\mathcal{A}(s^{\prime})) \psi_{\mathcal{A}(s^{\prime})}(s^{\prime}-t)\|_{L(H)} \mathrm{d} t \\
				& \leq \int_{|t-s| \geq 2|s^{\prime}-s|}\|\mathcal{A}(s) \psi_{\mathcal{A}(s)}(s-t)-(\mathcal{A}(s^{\prime})) \psi_{\mathcal{A}(s^{\prime})}(s-t)\|_{L(H)} \mathrm{d} t \\
				& \quad+\int_{|t-s| \geq 2|s^{\prime}-s|}\|(\mathcal{A}(s^{\prime})) \psi_{\mathcal{A}(s^{\prime})}(s-t)-(\mathcal{A}(s^{\prime}))\psi_{\mathcal{A}(s^{\prime})}(s^{\prime}-t)\|_{L(H)} \mathrm{d} t=: I_1+I_2.
			\end{aligned}
			$$
			The second term $I_2$ can be addressed exactly as the proof of \eqref{4.6}. For the first term $I_1$, writing
			$$
			\begin{aligned}
				\mathcal{A}(s)\psi_{\mathcal{A}(s)}(s-t)&-(\mathcal{A}(s^{\prime}))\psi_{\mathcal{A}(s^{\prime})}(s-t)
				\\&=\frac{1}{2 \pi i} \int_{\Gamma} \lambda^{\alpha}e
				^{\lambda (s-t)}(\lambda^{\alpha}+\mathcal{A}(s))^{-1}-(\lambda^{\alpha}+ \mathcal{A}(s^{\prime}))^{-1} \mathrm{d} \lambda.
			\end{aligned}
			$$
			By a similar argument in Lemma \ref{4.3}, we have 
			$$
			\begin{aligned}
				\|\mathcal{A}(s)\psi_{\mathcal{A}(s)}(s-t)-\mathcal{A}(s^{\prime})\psi_{\mathcal{A}(s^{\prime})}(s-t)\|_{L(H)} & \leq C \int_{\Gamma}  e^{\operatorname{Re}\lambda(s-t)} \omega(|s-s^{\prime}|) |d\lambda| \\
				& \leq  \frac{C\omega(|s-s^{\prime}|)}{s-t}.
			\end{aligned}
			$$
			Therefore, we obtain
			$$
			\begin{aligned}
				\int_{|t-s| \geq 2|s^{\prime}-s|}\|\mathcal{A}(s)\psi_{\mathcal{A}(s)}(s-t)&-\mathcal{A}(s^{\prime})\psi_{\mathcal{A}(s^{\prime})}(s-t)\|_{L(H)} \mathrm{d} t \\&\leq C \int_{|t-s| \geq 2|s^{\prime}-s|} \frac{\omega(|s-s^{\prime}|)}{s-t} \mathrm{~d} t.
			\end{aligned}
			$$
			In the case of $s\geq s'$, we have 
			$$
			\int_{|t-s| \geq 2|s^{\prime}-s|}\frac{\omega(|s-s^{\prime}|)}{s-t} \mathrm{~d} t=\int_{0}^{2s'-s}\frac{\omega(s-s^{\prime})}{s-t} \mathrm{~d} t\leq \int_{2s-2s'}^{s}\frac{\omega(r)}{r} \mathrm{~d} r\leq C,
			$$
			where applied the fact that $s-s^{\prime}\leq s-t$ and $\omega$ is non-decreasing to write the first inequality. The case $s^{\prime}>s$ is treated similarly. The proof is completed.
		\end{proof}
		We now present the main proof of the Theorem \ref{T1}.
		\begin{proof}[Proof of Theorem \ref{T1}]
			We shall first consider the case for $f \in C_c^{\infty}(0, \tau ; H)$ and $u_0\in(H, \mathcal{D}(A(0)))_{1-\frac{1}{\alpha p}, p}$ for $\alpha>\frac{1}{p}$. 
			
			From Lemma \ref{4.1}, it is clear that
			$(I-Q) A(\cdot) u^{\delta}(\cdot)=Lf^{\delta}(\cdot)+Ru_0.$ 
			Lemma \ref{4.2} allows us to rewrite $A(\cdot) u^{\delta}(\cdot)$ into 
			$$
			A(\cdot) u^{\delta}(\cdot)=(I-Q)^{-1}(L f^{\delta}(\cdot)+Ru_0).
			$$
			Using the $L_p$-boundedness of $(I-Q)^{-1}$ in Lemma \ref{4.2}, together with the results of Lemma \ref{4.3} and Lemma \ref{4.5}, we obtain
			$$
			\begin{aligned}
				\|A(\cdot)u(\cdot)\|_{L_p(0,\tau;H)}&\leq C\|A(\cdot)u^{\delta}(\cdot)\|_{L_p(0,\tau;H)}\leq C\|L f^{\delta}(\cdot)+Ru_0\|_{L_p(0,\tau;H)}
				\\&\leq C(\|f\|_{L_p(0,\tau;H)}+\|u_0\|_{(H, \mathcal{D}(A(0)))_{1-\frac{1}{\alpha p}, p}}).
			\end{aligned}
			$$
			From the boundedness of $A(t)^{-1}$ and problem (FP), we also have  
			$$
			\|u \|_{L_p(0,\tau;H)}+\|\partial_t^{\alpha} u \|_{L_p(0,\tau;H)}\leq C(\|f\|_{L_p(0,\tau;H)}+\|u_0\|_{(H, \mathcal{D}(A(0)))_{1-\frac{1}{\alpha p}, p}}).
			$$
			We thus establish the main estimate \eqref{Est}.
			
			Now, let $f\in L_p(0, \tau ; H)$ and suppose $f_n$ is a sequence in $C_c^{\infty}(0, \tau ; H)$ that converges to $f$ in $L_p(0, \tau ; H)$ and pointwise almost everywhere. Let $u_n$ represent the solution to the problem (FP) with non-homogeneous term $f_n$. We apply \eqref{Est} to $u_n-u_m$ and observe that there exists $u\in L_p(0, \tau ; H)$ with $\partial_t^{\alpha} u \in L_p(0, \tau ; H)$ and $v \in L_p(0, \tau ; H)$ such that
			$$
			u_n \xrightarrow{L_p} u \quad \partial_t^{\alpha} u_n \xrightarrow{L_p} \partial_t^{\alpha} u \quad \text { and } \quad A(\cdot) u_n(\cdot) \xrightarrow{L_p} v.
			$$
			By choosing a subsequence, these limits can be assumed to hold pointwise almost everywhere. Since operator $A(t)$ is closed for an each fixed $t$, it follows that $v(\cdot)=A(\cdot) u(\cdot)$. Taking the limits in the equation
			$$
			\partial_t^{\alpha} u_n(t)+A(t) u_n(t)=f_n(t), \quad u_n(0)=u_0,
			$$
			implies that
			$$
			\partial_t^{\alpha} u(t)+A(t) u(t)=f(t), \quad u(0)=u_0
			$$
			for a.e. $t \in(0, \tau)$, which is in the $L_p$ sense. Therefore, $u$ can be identified as a solution to problem (FP). Moreover, \eqref{Est} is satisfied for $u$ by convergence, and the uniqueness is derived from the priori estimate \eqref{Est}. The case of $\alpha\leq \frac{1}{p}$ can be proved by the same arguments immediately.
		\end{proof}
		\section{A non-symmetric counterexample} \label{sect:5}
		
		In this section, we construct non-symmetric non-autonomous forms that fail to have maximal regularity.
		We first introduce a fundamental special function that plays a key role in our analysis. The Mittag-Leffler function, for $0<\alpha<1$ and $\beta\in \mathbb{C}$ defined by
		$$
		E_{\alpha,\beta}(z)=\sum_{k=0}^{\infty} \frac{z^k}{\Gamma(\alpha k+\beta)}, \quad z \in \mathbb{C},
		$$
		which satisfies the properties $$
		\partial_t^{\alpha}E_{\alpha,1}(zt^{\alpha})=z E_{\alpha,1}(zt^{\alpha}),~~\text{ and }~ \frac{d}{dt}E_{\alpha,1}(zt^{\alpha})=zt^{\alpha-1}E_{\alpha,\alpha}(zt^{\alpha}).
		$$
	    For $0<\alpha<1$ and $|\arg z| \leq \frac{\pi \alpha}{2}$,
		the asymptotic expansion of $E_{\alpha,\beta}(z)$ for any positive integer $m$ is given by
		$$
		E_{\alpha, \beta}(z)=\frac{1}{\alpha} z^{(1-\beta) / \alpha} \mathrm{e}^{\mathrm{z}^{1 / \alpha}}-\sum_{k=1}^m \frac{z^{-k}}{\Gamma(\beta-k \alpha)}+O\left(|z|^{-m-1}\right),|z| \rightarrow \infty.
		$$
		The content above can be found in \cite[section 4.5.2.]{Zhou}.
		\begin{Lemma}\label{Lemma5.1}
			For all $\mu\in \mathbb{R_+}$ and $0< s<t\leq \tau$, it holds
			\begin{enumerate}[label=(\roman*)]
				\item $|E_{\alpha,1}(i^{\alpha}\mu)|\leq M$ for some $M>0$.
				\item $|E_{\alpha,1}(i^{\alpha}t^{\alpha} \mu)-E_{\alpha,1}(i^{\alpha}s^{\alpha} \mu)| \leq C(\mu(t-s)^{\alpha}+\mu^{\frac{1}{\alpha}}(t-s))$ for some $C>0$.
			\end{enumerate}
		\end{Lemma}
		\begin{proof}
Let $\mu\rightarrow \infty$, then the following asymptotic expansion holds
			$$
			E_{\alpha, 1}(i^{\alpha} \mu)=\frac{1}{\alpha} \mathrm{e}^{i \mu^{\frac{1}{\alpha}}}-\sum_{k=1}^m \frac{(i^{\alpha} \mu)^{-k}}{\Gamma(1-k \alpha)}+O\left(| \mu|^{-m-1}\right),~~ \mu  \rightarrow \infty.
			$$
			We observe that $E_{\alpha, 1}(i^{\alpha} \mu)$ is an oscillatory function at disk with radius $\frac{1}{\alpha}$. This leads to the result that $|E_{\alpha, 1}(i^{\alpha} \mu)|= \frac{1}{\alpha}$ as $\mu \rightarrow \infty$. Since $|E_{\alpha, 1}(i^{\alpha} \mu)|$ is a continuous function, the conclusion of (i) follows. To prove (ii), we have
			$$
			|E_{\alpha,1}(i^{\alpha}t^{\alpha} \mu)-E_{\alpha,1}(i^{\alpha}s^{\alpha} \mu)|=|\int_{s}^{t}\mu r^{\alpha-1} E_{\alpha,\alpha}(i^{\alpha} r^{\alpha}\mu)dr|.
			$$
            Observe that
            $$
            E_{\alpha, \alpha}(i^{\alpha}r^{\alpha} \mu)=\frac{1}{\alpha} (i^{\alpha}r^{\alpha} \mu)^{(1-\alpha)/\alpha}\mathrm{e}^{ir \mu^{\frac{1}{\alpha}}}-\sum_{k=1}^m \frac{(i^{\alpha}r^{\alpha} \mu)^{-k}}{\Gamma(\alpha-k \alpha)}+O\left(|r^{\alpha} \mu|^{-m-1}\right),~~ r^{\alpha}\mu  \rightarrow \infty.
            $$
			Then for a sufficiently large $K>0$, as $r^{\alpha}\mu\in (0,K]$, by virtue of definition we have $|E_{\alpha,\alpha}(i^{\alpha} r^{\alpha}\mu)|\leq C$, which implies
			$$
			|E_{\alpha,1}(i^{\alpha}t^{\alpha} \mu)-E_{\alpha,1}(i^{\alpha}s^{\alpha} \mu)|\leq  C\mu (t^{\alpha}-s^{\alpha})\leq   C\mu(t-s)^{\alpha}.
			$$
			As $r^{\alpha}\mu\in (K,\infty)$, we have
            $$
            \begin{aligned}
                |E_{\alpha,1}(i^{\alpha}t^{\alpha} \mu)-E_{\alpha,1}(i^{\alpha}s^{\alpha} \mu)|&\leq2\int_{s}^{t}\mu r^{\alpha-1} \frac{1}{\alpha} (r^{\alpha} \mu)^{(1-\alpha)/\alpha}dr\\
                &\leq C\mu^{\frac{1}{\alpha}}(t-s).
            \end{aligned}
            $$
		   This  completes the proof. 
		\end{proof}
		\begin{Lemma}\label{Lemma5.0}
			Let $\partial^{\alpha}_t u \in L_2(0,\tau;H)$, $v \in C^{\infty}(0,\tau;H)$, then $\partial^{\alpha}_t ( uv)\in L_2(0,\tau;H)$ in the sense of distribution
			$$
			\partial^{\alpha}_t (v(t)u(t))=	v(t)\partial^{\alpha}_t u(t)+F(u(t),v(t)),
			$$
			where 
			$$
			F(u(t),v(t))=\int_{0}^{t}-\dot{k}(t-s)(v(t)-v(s))u(s)ds+k(t)(v(t)-v(0))u(0).
			$$
		\end{Lemma}
		\begin{proof}
			Since $\partial^{\alpha}_t u \in L_2(0,\tau;H)$, one can pick up a sequence  $u_n \in C_c^{\infty}(0,\tau;H)$ such that
			$$
			\|u -u_n \|_{L_2(0,\tau;H)}\rightarrow 0,~~ \text{ and }~~ \|\partial^{\alpha}_t u -\partial^{\alpha}_t u_n \|_{L_2(0,\tau;H)}\rightarrow 0,~~ \text{ as } n\rightarrow\infty.
			$$
			Observe that
			$$
			v(t)\partial^{\alpha}_t u_n(t)=\int_{0}^{t}k(t-s)(v(t)-v(s))\frac{du_n(s)}{ds}ds+\int_{0}^{t}k(t-s)v(s)\frac{du_n(s)}{ds}ds=:J_1+J_2,
			$$
			where
			$$
			\begin{aligned}
				J_1=&\int_{0}^{t}k(t-s)\frac{d(v(t)-v(s))u_n(s)}{ds}ds-\int_{0}^{t}k(t-s)u_n(s)\frac{d(v(t)-v(s))}{ds}ds=:J_{11}+J_{12}.
			\end{aligned}
			$$
			Therefore, the formula of integration by parts shows that
			$$\begin{aligned}
				J_{11} 
				=&\left[k(t-s)(v(t)-v(s))u_n(s)\right]^t_0-\int_{0}^{t}-\dot{k}(t-s)(v(t)-v(s))u_n(s)ds,
			\end{aligned}$$
			and  then
			$$
			J_{2}=\int_{0}^{t}k(t-s)\frac{du_n(s)v(s)}{ds}ds-\int_{0}^{t}k(t-s)u_n(s)\frac{dv(s)}{ds}ds=\partial^{\alpha}_t (v(t)u_n(t))-J_{12}.
			$$
			Putting above arguments together, we get
			$$
			\partial^{\alpha}_t (v(t)u_n(t))=	v(t)\partial^{\alpha}_t u_n(t)+F(u_n(t),v(t)).
			$$
			We thus obtain that 
			$$
			\begin{aligned}
				&\|\partial^{\alpha}_t (v u_n )-\partial^{\alpha}_t (v u )\|_{L_2(0,\tau;H)}+\|v \partial^{\alpha}_t u_n -v \partial^{\alpha}_t u \|_{L_2(0,\tau;H)}+\|F(u_n ,v )-F(u ,v )\|_{L_2(0,\tau;H)}
				\\ \rightarrow & 0,~~~ \text{ as } n\rightarrow 0.
			\end{aligned}
			$$
			This lemma is proved.
		\end{proof}
		We will now begin constructing our counterexample. Assume that $u $ is a solution to Cauchy problem (FP), it follows that
		$$\langle\partial_t^{\alpha}u(t), v\rangle_{V^{\prime}, V}+a(t, u(t), v)=(f(t), v)_H$$ for all $v \in V$ and $t \in(0, \tau]$.
		The solution determines the values of the form $a(t, \cdot, \cdot)$ over the product space $\{ u\} \times V$ as
		\begin{equation}\label{EQ5.1}
			a(t, c \cdot u(t), v)=c\left[(f(t) , v)_H-\langle\partial_t^{\alpha}u(t), v\rangle_{V^{\prime},V}\right]
		\end{equation}
		for all $c \in \mathbb{C}$ and $v \in V$. Recall that $u\in L_2(0,\tau;V)$ with $\partial^{\alpha}_tu \in L_2(0,\tau;V^{\prime})$ solving (FP) has maximal $L_2$-regularity if and only if $\partial^{\alpha}_tu \in L_2(0,\tau;H)$.  Following the methodological framework proposed by Fackler \cite{Fackler2}, we construct a counterexample for problem (FP) using the following two key steps. First, we deliberately select a badly behaved function $u$ that exhibits undesirable regularity properties. Subsequently, we systematically derive an operator structure $A(t)$ that satisfies both the evolution equation (\ref{EQ5.1}) and the foundational assumptions (\ref{A0}).  
		
		\subsection{Badly behaved function and extending forms}
		
		We select $H=L_2([0,1])$ and $V=L_2([0,1] ; w d \lambda)$, where $w:[0,1] \rightarrow \mathbb{R}_{\geq 1}$ is a measurable locally bounded weight. Then $V^{\prime}=L_2\left([0,1] ; w^{-1} d \lambda\right)$ and we obtain the Gelfand triple structure $V \hookrightarrow H \hookrightarrow$ $V^{\prime}$ via standard embeddings.
        
		In the sequel, we can choose the badly behaved function $u(t)=c(x) E_{\alpha,1} (i^{\alpha}t^{\alpha} \varphi(x))$ with
        \[
		w(x)=x^{-a},\quad \varphi(x)=x^{-b},\quad c(x)=x^c, \quad \text{for } a,b,c\geq0,
		\]
        where these functions are required to satisfy $\partial_t^{\alpha}u\notin L_2(0,\tau ; H) $, $u \in L_2(0, \tau ; V)$ and $\partial_t^{\alpha}u  \in L_2(0, \tau ; V^{\prime})$, i.e.,
\begin{equation}\label{EQB}\tag{B}
			\begin{aligned}	\|\partial_t^{\alpha} u(t)\|_H^2=& \int_0^1|i^{\alpha}\varphi(x)c(x)E_{\alpha,1}(i^{\alpha}t^{\alpha}\varphi(x))|^2 d x\\
				= & m_1(t)\int_0^1|\varphi(x) c(x)|^2 d x=m_1(t)\|\varphi c\|_H^2=\infty,
			\end{aligned}
		\end{equation}
and
\begin{equation}\label{EQC}\tag{C}
			\begin{aligned}
				\|u(t)\|^2_V & =m_2(t) \int_0^1|c(x)|^2 w(x) d x=m_2(t)\|c\|_V^2<\infty, \\
				\|\partial_t^{\alpha}u(t)|_{V'}^2& =m_3(t) \int_0^1|\varphi(x) c(x)|^2 w^{-1}(x) d x=m_3(t)\|\varphi c\|_{V^{\prime}}^2<\infty ,
			\end{aligned}
		\end{equation}
  where $m_j(t)>0$, $j=1,2,3$ are bounded positive functions by virtue of the asymptotic expansion of $|E_{\alpha,1}(i^{\alpha}t^{\alpha}\varphi(x))|$.
		
		Now consider the associated form (\ref{EQ5.1}). For $a(t,\cdot,\cdot)$ to be uniformly coercive, it is required that 
		$$
		\begin{aligned}
			\operatorname{Re} a(t, u(t), u(t)) =\operatorname{Re}(f(t) , u(t))_H-\operatorname{Re}\langle\partial_t^{\alpha}u(t), u(t)\rangle_{V^{\prime}, V} \geq \gamma\|u(t)\|_V^2,
		\end{aligned}
		$$
		for all $t \in(0, \tau]$. Observe that, from the denseness embedding $H\rightarrow V^{\prime}$, assuming that 
		\begin{equation}\label{EQD}\tag{D}
			\int_{0}^{1}\varphi(x)|c(x)|^2dx=\|\varphi(x)^{1/2}c(x)\|^2_H<\infty,
		\end{equation}
		we have
		$$
		\begin{aligned}
			\operatorname{Re}\langle\partial_t^{\alpha}u(t), i^{\alpha-1}u(t)\rangle_{V^{\prime}, V}=&\operatorname{Re}\int_{0}^{1}i^{1-\alpha}\partial_t^{\alpha}u(t,x)\overline{u(t,x)}dx\\ =&\operatorname{Re}\int_{0}^{1}i\varphi(x)|c(x)|^2|E_{\alpha,1}(i^{\alpha}t^{\alpha}\varphi(x))|^2dx=0.
		\end{aligned}
		$$ Hence we choose $f(t)=i^{\alpha-1}u(t)$ for simplicity and then 
		$$
		\|u(t)\|_H^2=\int_0^1|u(t, x)|^2 d x=m_4(t)\int_0^1|c(x)|^2 d x=m_4(t)\|c\|_H^2,
		$$
		for some $m_4(t)>0$ and
		\begin{equation}\label{EQ5.2}
			\begin{aligned}
				\operatorname{Re} a(t, u(t), u(t)) 
				= \|u(t)\|_H^2\geq \gamma \|u(t)\|_V^2,
			\end{aligned}
		\end{equation}
		with $\gamma\leq\|u(t)\|_H^2/\|u(t)\|_V^2=(m_4(t)\|c\|_H^2)/(m_2(t)\|c\|_V^2)$ for all $t\in [0,\tau]$.
		Moreover, the uniform boundedness of (\ref{EQ5.1}) on \( \{ u \} \times V \) is a consequence of \( u \in L_{\infty}(0, \tau; H) \) and \( \partial_t^{\alpha} u \in L_{\infty}(0, \tau; V^{\prime}) \). 
        
  \begin{remark}
    To contextualize Fackler's approach \cite{Fackler2}, we observe that for problem (P), the construction of counterexamples hinges on employing a badly behaved function \( u(t) = c(x) \exp(it\varphi(x)) \), specifically designed to violate regularity assumptions while satisfying structural conditions akin to (\ref{EQB}) and (\ref{EQC}). However, the exponential ansatz becomes inadequate when addressing problem (FP), due to the intrinsic non-locality and memory effects inherent in fractional differential operators.
To overcome this limitation, we introduce a fractional generalization of the exponential function by adopting the Mittag-Leffler function. This special function preserves the oscillatory character of \(  \exp(it\varphi(x)) \) while accounting for the fractional order $\alpha$, thereby enabling the construction of counterexamples that expose analogous irregularities in the fractional setting. The choice of  \( E_{\alpha,1}(i^{\alpha}t^{\alpha}\varphi(x)) \)
  as the kernel function ensures compatibility with the solution structure of fractional evolution equations, making it a natural candidate for extending Fackler's methodology to the fractional framework.
        \end{remark} 
        To construct a counterexample, it is necessary to extend these sesquilinear forms to \( V \times V \).
		

		Let $b(t;u,v)$ be the form defined by \eqref{EQ5.1}. Our aim is to extend  $b(t;u,v):\{ u\} \times V\rightarrow \mathbb{C}$ to the form $a(t;u,v):V \times V\rightarrow \mathbb{C}$.
		For a fixed $t$, consider the form $b(t;\cdot,\cdot)-\frac{\gamma}{2}(\cdot,\cdot)_V$, the associated operator is $T_t:U\rightarrow V$. Then from \eqref{EQ5.2} we have $$\operatorname{Re}(T_tu,u)=\operatorname{Re}b(t;u,u)-\operatorname{Re}\frac{\gamma}{2}(u,u)_V>0.$$ To keep the coercivity invariant, we need search extensions $\hat{T}: V \rightarrow V$ with the numerical range
		$$
		W(T):=\left\{(T u, u)_V: u \in U,\|u\|_V=1\right\}\subset \{z \in \mathbb{C}: \operatorname{Re} z \geq 0\}  .
		$$
		We introduce the following three lemmas about this extension operator. The proofs can be found in \cite{Fackler2}, but we present here for completeness.
		\begin{Lemma}\label{Lemma5.2}
			Consider $U$ be a closed subspace of the Hilbert space $V$, and let $T: U \rightarrow U$ be a bounded linear operator with $W(T) \subset$ $\{z \in \mathbb{C}: \operatorname{Re} z \geq 0\}$. Then there exists a bounded extension $\hat{T}: V \rightarrow V$ preserving the operator norm $\|\hat{T}\|=\|T\|$ while maintaining the numerical range condition $W(\hat{T}) \subset\{z \in \mathbb{C}: \operatorname{Re} z \geq 0\}$.
		\end{Lemma}
		\begin{proof}
			Using the orthogonal complement \( U^\perp \) of \( U \) in \( V \), we define the trivial extension
			$$
			\hat{T}: U \oplus U^\perp \ni u + u^\perp \mapsto T u.
			$$
			Observe that this definition ensures: $\|\hat{T}\| = \|T\|$, as \( \hat{T} \) acts as \( T \) on \( U \) and is zero on \( U^\perp \). Furthermore, the numerical range \( W(\hat{T}) \) is the convex hull of \( W(T) \) and \( \{0\} \). Since
			$
			W(T) \subseteq \{z \in \mathbb{C} \mid \operatorname{Re} z \geq 0\},
			$
			it follows that
			$
			W(\hat{T}) \subseteq \{z \in \mathbb{C} \mid \operatorname{Re} z \geq 0\}.
			$
		\end{proof}
		\begin{Lemma}\label{Lemma5.3}
			Consider a one-dimensional subspace $U\subset V$, and let $T: U \rightarrow V$ be a bounded linear operator with $W(T) \subset\{z \in \mathbb{C}: \operatorname{Re} z \geq 0\}$. Then there exists a bounded extension $\hat{T}: V \rightarrow V$ satisfying $\|\hat{T}\| \leq \sqrt{2}\|T\|$ and $W(\hat{T}) \subset$ $\{z \in \mathbb{C}: \operatorname{Re} z \geq 0\}$.
		\end{Lemma}
		\begin{proof}
			Let $W=span\{U,TU\}$. If $\dim W=1$, the result directly follows from Lemma \ref{Lemma5.2}. When $\dim W=2$, we construct an orthonormal system $\left(e_1, e_2\right)$ of $W$ with $e_1 \in U$. Define a linear extension $S$ : $W \rightarrow W$ of $T$ satisfying $\left(S e_2 , e_1\right)=-\overline{\left(T e_1 , e_2\right)}$ and $\left(S e_2 , e_2\right)=0$. Then for $w=\lambda_1 e_1+\lambda_2 e_2 \in W$,
			we have
			$$
			(S w , w)=\left|\lambda_1\right|^2\left(T e_1 , e_1\right)+\left|\lambda_2\right|^2\left(S e_2 , e_2\right)+\lambda_1 \overline{\lambda_2}\left(T e_1 , e_2\right)+\overline{\lambda_1} \lambda_2\left(S e_2 , e_1\right),
			$$
			and $\operatorname{Re}(S w , w)=\left|\lambda_1\right|^2 \operatorname{Re}\left(T e_1 , e_1\right),$ 
			which implies $W(\hat{T}) \subset$ $\{z \in \mathbb{C}: \operatorname{Re} z \geq 0\}$ and 
			$$
			\|S w\| \leq\left(\left|\lambda_1\right|+\left|\lambda_2\right|\right)\|T\| \leq \sqrt{2}\|T\|\|w\| .
			$$
			Then by applying Lemma \ref{Lemma5.2} the extension $\hat{T}:V\rightarrow V$  to  $S $ follows.
		\end{proof} 
		Applying Lemma \ref{Lemma5.3} to the operator $T_t$ associated with $b(t;\cdot,\cdot)-\frac{\gamma}{2}(\cdot,\cdot)_V$, yields the extension $\hat{T}:V\rightarrow V$ with the extension form $a(t;\cdot,\cdot)$.
		\begin{Lemma}\label{Lemma5.4}
			Consider a one-dimensional subspace $U\subset V$ and sesquilinear form $b: U \times V \rightarrow \mathbb{C}$ that admits positive constants $\gamma ,M$ such that for all $u \in U$ and $v \in V$:
			$$
			|b(u, v)| \leq M\|u\|_V\|v\|_V, \quad \operatorname{Re} b(u, u) \geq \gamma\|u\|_V^2 .
			$$
			Then there exists an extension of $b$ to a sesquilinear form $a: V \times V \rightarrow \mathbb{C}$, where $a$ satisfies
			$$
			\begin{aligned}
				|a(u, v)|   \leq\left[\sqrt{2}\left(M+\frac{\gamma}{2}+2 \gamma^{-1}\left(M+\frac{\gamma}{2}\right)^2\right)+\frac{\gamma}{2}\right]\|u\|_V\|v\|_V,~~ 
				\operatorname{Re} a(u, u)   \geq \frac{\gamma}{2}\|u\|_V^2,
			\end{aligned}
			$$
			for all $u, v \in V$.
		\end{Lemma}
		For our specific function $u(t, x)=c(x) E_{\alpha,1}( i^\alpha t^\alpha \varphi(x))$ and $f(t)=i^{\alpha-1}u(t)$ we obtain
		$$
		\begin{aligned}
			b(t, u(t), v) & =(f(t) , v)_H-\langle\partial_t^{\alpha}u(t), v\rangle_{V^{\prime}, V}=\int_0^1 i^{\alpha-1}u(t, x) \overline{v(x)} d x-\int_0^1 \partial_t^{\alpha}u(t, x) \overline{v(x)} d x \\
			& =\left(w^{-1}(i^{\alpha-1}u(t)-\partial_t^{\alpha}u(t)) , v\right)_V .
		\end{aligned}
		$$
		To extend this form to $V\times V$, we follow the step in the proof of Lemma \ref{Lemma5.2} and Lemma \ref{Lemma5.3}. Let $e_1=u(t) /\|u(t)\|_V$ denote the normalized basis vector in a Hilbert space $V$ and consider the operator $T_t$ associated to $b(t, \cdot, \cdot)-\frac{\gamma}{2}(\cdot,\cdot)_V$, then 
		$$
		T_t e_1=\frac{i^{\alpha-1}u(t)-\partial_t^{\alpha}u(t)}{w\|u(t)\|_V}-\frac{\gamma}{2} \frac{u(t)}{\|u(t)\|_V}.
		$$   
		For the orthogonal projection of $T_t u(t)$ onto the complement of $e_1$ is $T_t u(t)-\left(T_t u(t), e_1\right)_V e_1$, which is equal to
		$$
		\begin{aligned}
			& T_t u(t)-\left[-\frac{\gamma}{2}+\frac{1}{\|u(t)\|_V^2}\left(\|u(t)\|_H^2-\langle\partial_t^{\alpha}u(t), u(t)\rangle_{V^{\prime},V}\right)\right] u(t) \\
			&=w^{-1}\partial_t^{\alpha}u(t)+(w^{-1}-\frac{1}{\|u(t)\|_V^2}\left(\|u(t)\|_H^2-\langle\partial_t^{\alpha}u(t), u(t)\rangle_{V^{\prime},V}\right)) u(t) =: k(t).
		\end{aligned}
		$$
		Since $|E_{\alpha,1}(i^\alpha t^\alpha \varphi(x))|^2>0$ almost everywhere by the analyticity of function $E_{\alpha,1}(z)$, this implies $\|u(t)\|_V,\|k(t)\|_V>0$.  Then there are normalization factors $n_1=\|u(t)\|_V$ and $n_2=\|k(t)\|_V$ associated with $u(t)$ and $k(t)$ in $V$. Therefore, the orthogonal projection $P_t: V \rightarrow \operatorname{span}\left\{u(t), T_t u(t)\right\}$ is represented by
		\begin{equation}\label{EQ5.3}
			P_t v=\frac{1}{\|u(t)\|_V^2}(v , u(t))_V u(t)+\frac{1}{\|k(t)\|_V^2}(v , k(t))_V k(t) .
		\end{equation}
		One then has for $y_1, y_2 \in V$,
		\begin{equation}\label{EQ5.4}
			\begin{aligned}
				\left(S_t P_t y_1 , P_t y_2\right)_V= & \frac{1}{\|u(t)\|_V^4}\left(y_1 , u(t)\right)_V \overline{\left(y_2 , u(t)\right)_V}\left(T_t u(t) , u(t)\right)_V \\
				& +\frac{1}{\|u(t)\|_V^2\|k(t)\|_V^2}\left(y_1 , u(t)\right)_V \overline{\left(y_2 , k(t)\right)_V}\left(T_t u(t) , k(t)\right)_V \\
				& -\frac{1}{\|u(t)\|_V^2\|k(t)\|_V^2}\left(y_1 , k(t)\right)_V \overline{\left(y_2 , u(t)\right)_V\left(T_t u(t) , k(t)\right)_V}.
			\end{aligned}
		\end{equation}
		Combining all the components, the extended forms in Lemma \ref{Lemma5.4} are expressed as follows 
		$$
		a\left(t, y_1, y_2\right)=(\hat{T}_t y_1 , y_2)_V+\frac{\gamma}{2}\left(y_1 , y_2\right)_V=\left(S_t P_t y_1 , P_t y_2\right)_V+\frac{\gamma}{2}\left(y_1 , y_2\right)_V.
		$$
		\subsection{\texorpdfstring{The H\"older exponent of the form}{The Holder exponent of the form}}
		
		In the following we focus on the H\"{o}lder regularity of the extending forms.  Note that,  multiplying two bounded Hölder continuous functions pointwise retains H\"{o}lder continuity with the original exponent, in particular,  an analogous result holds in the inverse of a positive H\"{o}lder continuous function as well as its norms.
		Considering this, we observe from the detailed forms in (\ref{EQ5.3}) and (\ref{EQ5.4}) that the form $a$ is $\beta$-H\"{o}lder continuous when the functions $u:(0, \tau] \rightarrow V,~ w^{-1} u:(0, \tau] \rightarrow V$ and $\partial_t^{\alpha}u:(0, \tau] \rightarrow V^{\prime}$ all exhibit $\beta$-H\"{o}lder continuity‌. Next, we address these functions. For the first function, we have
		$$
		\begin{aligned}
			\|u(t)-u(s)\|_V^2 & =\int_0^1 w(x)|u(t, x)-u(s, x)|^2 d x \\
			& =\int_0^1 w(x)|c(x)|^2|E_{\alpha,1}(i^\alpha t^\alpha  \varphi(x))-E_{\alpha,1}(i^\alpha s^\alpha  \varphi(x))|^2 d x.
		\end{aligned}
		$$
		We now aim to maximize the H\"{o}lder exponent in the preceding expression while ensuring that conditions (\ref{EQB}), (\ref{EQC}), and (\ref{EQD}) remain valid. Within the family of functions defined by
		\[
		w(x)=x^{-a},\quad \varphi(x)=x^{-b},\quad c(x)=x^c, \quad \text{for } a,b,c\geq0,
		\]
		the optimal H\"{o}lder exponent is obtained by selecting \( a=b=\frac{3}{2} \) and \( c=1 \). It is straightforward to verify that this choice satisfies conditions  (\ref{EQB}), (\ref{EQC}) and (\ref{EQD}) . Moreover,  Lemma \ref{Lemma5.1} implies that
		 $$|E_{\alpha,1}(i^{\alpha}t^{\alpha} \varphi(x))-E_{\alpha,1}(i^{\alpha}s^{\alpha} \varphi(x))| \leq C(|t-s|^{\alpha} \varphi(x)+C|t-s|\varphi(x)^{\frac{1}{\alpha}}) ,$$
		   it follows that for $|t-s| \leq 1$,
		$$
		\begin{aligned}
			\|u(t)-u(s)\|_V^2  \leq& 4M \int_0^{|t-s|^{2 \alpha / 3}} x^{1 / 2} d x+C^2|t-s|^{2\alpha} \int_{|t-s|^{2\alpha / 3}}^1 x^{1 / 2} x^{-3} d x \\&+C^2|t-s|^2 \int_{|t-s|^{2\alpha / 3}}^1 x^{1 / 2} x^{-\frac{3}{\alpha}} d x\\
			 \leq& C|t-s|^{\alpha}+C|t-s|^{2\alpha}|t-s|^{-\alpha}+C|t-s|^{2}|t-s|^{\alpha-2}\\=&C|t-s|^{\alpha}.
		\end{aligned}
		$$
		For $|t-s| \geq 1$ the same H\"{o}lder estimate hold trivially. This implies that $u:(0, \tau] \rightarrow V$ is $\frac{\alpha}{2}$-H\"{o}lder continuous. Moreover, since $\left\|w^{-1}(u(t)-u(s))\right\|_V \leq\|u(t)-u(s)\|_V$, and  $\partial_t^{\alpha}u$ has the same expression as for $u$,  hence these two functions have the same H\"{o}lder exponent.
		
		We next give a proof of the Theorem \ref{T2}.
		\begin{proof}[Proof of Theorem \ref{T2}]
			Let $u(t, x) = c(x) E_{\alpha,1}(i^\alpha t^\alpha  \varphi(x))$. Let $\eta : [0, \tau] \to \mathbb{R}$  be a a smooth cut-off function satisfying $0 \leq \eta \leq 1$, $\eta(0) = 0$, and $\eta(t) = 1$ for all $t \geq \frac{\tau}{2}$. Now set $w = \eta u$, it follows that $w$ satisfies $w(0) = 0$, $\partial^{\alpha}_t w  \in L_2(0, \tau ; V)$, and $w \in L_2(0,\tau ; V)$, the Lemma \ref{Lemma5.0} shows the following identity
			$$\begin{aligned}
				\partial^{\alpha}_t w(t) + \mathcal{A}(t) w(t) &= \eta(t) \partial^{\alpha}_t u(t) + \eta(t) \mathcal{A}(t) u(t) + \partial^{\alpha}_t \eta(t) u(t) + F(u(t), \eta(t)) \\
				&= \eta(t)i^{\alpha-1} u(t) + \partial^{\alpha}_t \eta(t) u(t) + F(u(t), \eta(t)),
			\end{aligned}
			$$
			where $\mathcal{A}(t)$ is the associated operator of the extended form. 	It is clear that the right-hand side belongs to $L_2(0, \tau; H)$ since $u$ does. However,  $\partial^{\alpha}_t w  = \partial^{\alpha}_t u $ on $[\frac{\tau}{2}, \tau]$, this means $\partial^{\alpha}_t w \notin L_2(0, \tau; H)$. Hence, $w \notin MR^{\alpha}_2((0, \tau])$. The proof is completed.
		\end{proof}

		\section*{Conflict of interest} 
		The authors declare that they have no conflict of interest.

		\section*{Acknowledgements}
		This work is supported by the National Natural Science Foundation of China (Nos. 12101142, 12471172). 

	\end{document}